\documentclass[a4paper,12pt]{amsart}
\usepackage{geometry}                		
\usepackage{graphicx}				
\usepackage{amssymb}

\usepackage{latexsym,exscale,enumerate,amsfonts,amssymb,mathtools}
\usepackage{amsmath,amsthm,amsfonts,amssymb,amscd,stmaryrd,textcomp}
\usepackage[normalem]{ulem}
\usepackage{young}
\usepackage{thmtools,xcolor}
\usepackage{easybmat}

\definecolor{colormy}{rgb}{0.8,0.05,0.05}
\definecolor{mycolor}{rgb}{0.25,0.99,0.25}

\usepackage[all]{xy}
\SelectTips{cm}{}

\usepackage{tikz}
\usetikzlibrary{decorations.markings}
\usetikzlibrary{decorations.pathreplacing}
\usetikzlibrary{arrows,shapes,positioning}
\tikzstyle directed=[postaction={decorate,decoration={markings,
    mark=at position #1 with {\arrow{>}}}}]
\tikzstyle rdirected=[postaction={decorate,decoration={markings,
    mark=at position #1 with {\arrow{<}}}}]

\usepackage{hyperref}

\newcommand{\Hom}{\mathrm{Hom}}
\newcommand{\End}{\mathrm{End}}

\newcommand{\ord}{\mathrm{ord}}
\newcommand{\mo}{\mathrm{mod }}

\def\C{{\mathbb C}}

\def\Z{{\mathbb Z}}

\def\pr{\mathrm{pr}}

\def\cha{\mathrm{char}}
\def\Tr{\mathrm{Tr}}

\theoremstyle{definition}
\newtheorem{thm}{Theorem}[section]
\newtheorem{cor}[thm]{Corollary}
\newtheorem{lem}[thm]{Lemma}
\newtheorem{prop}[thm]{Proposition}

\theoremstyle{definition}
\newtheorem{examplecounter}{Example}
\newtheorem{remarkcounter}{Remark}

\numberwithin{equation}{section}

\declaretheorem[style=definition,name=Definition,qed=$\blacktriangle$,numberlike=thm]{defn}

\title{Simple modules for Temperley-Lieb algebras and related algebras}

\author{Henning Haahr Andersen}

\email{h.haahr.andersen@gmail.com}

\date{}							

\begin{document}

\begin{abstract}
Let $k$ be an arbitrary field and let $q \in k\setminus\{0\}$. In this paper we use the known tilting theory for the quantum group $U_q(sl_2)$ to obtain the dimensions of simple modules for the Temperley-Lieb algebras $TL_n(q+q^{-1})$ and related algebras over $k$. Our main result is an algorithm which calculates the dimensions of simple modules for these algebras. We take advantage of the fact that $TL_n(q+q^{-1})$ is isomorphic to the endomorphism ring of the $n$'th tensor power of the natural $2$-dimensional module for the quantum group for $sl_2$. This algorithm is easy when the characteristic is $0$ and more involved in positive characteristic. We point out that our results for the Temperley-Lieb algebras contain a complete description of the simple modules for the Jones quotient algebras. Moreover, we illustrate how the same results lead to corresponding information about simple modules for the BMW-algebras and other algebras closely related with endomorphism algebras of families of tilting modules for $U_q(sl_2)$.
\end{abstract}

\maketitle

\section{Introduction}

Let $k$ be a field of characteristic $p \geq 0$ and $q \in k\setminus \{0\}$. Consider the Temperley-Lieb algebra $TL_n(q + q^{-1})$ on $n$ strands. If $q$ is not a root of unity or if $q = \pm 1$ and $p = 0$ then $TL_n(q + q^{-1})$ is semisimple and its representation theory is well understood, see e.g. \cite{M}, \cite{W}, \cite{AST1}. Also the case where $q$ is a root of unity and $p = 0$ (where $TL_n(q + q^{-1})$ is non-semisimple for $n \geq \ord(q^2)$) has been dealt with, see \cite{GW}, \cite{M}, \cite{ILZ}. So in this note we are mainly interested in the case where $p$ is positive and $q$ is a root of unity. Of special interest is the case $q =1$, i.e. $q + q^{-1} = 2 \in k$. Our first results give algorithms which determine the dimensions of all simple modules for $TL_n(q + q^{-1})$ for all $k$ and all $q$. An immediate consequence is that a tiny and especially easy part of these algorithms produces the dimensions of all simple modules for the Jones algebras.

Nest we turn to other algebras closely related to endomorphism algebras of tensor powers of higher dimensional tilting modules. This includes in particular the BMW-algebras (again over any field). We demonstrate how the decomposition of the tensor powers of the natural $2$-dimensional module into indecomposable summands also can be used in these cases to obtain dimensions of simple modules for such algebras.

It is wellknown that $TL_n(q + q^{-1})$ as the endomorphism ring for the $n$'th tensor power of the natural $2$-dimensional module $V_q$ for the quantum algebra $U_q(sl_2)$. This algebra is cellular and the standard cellular theory, \cite{GL},  \cite{AST1}  reveals that to determine the simple modules of $TL_n(q + q^{-1})$ is equivalent to determine the decomposition of $V_q^{\otimes n}$  into indecomposable summands. These summands are tilting modules and we take advantage of the explicitly known characters of indecomposable tilting modules for $U_q(sl_2)$ (see e.g. \cite{AT} for the characteristic $0$ case and \cite{E} for the $q=1$ case in prime characteristic) to deduce our algorithms.

The described method is easy to work with in characteristic zero and we shall start out by dealing with this case. Here we recover some of the results recently obtained by Iohara, Lehrer and Zhang \cite{ILZ}. Our way of attacking the problem (exploring the $sl_2$-side of the theory) is ''dual" to theirs. In prime characteristic our method works in much the same way although the algorithm becomes more elaborate.

When $q$ is a root of unity the Temperley-Lieb algebras $TL_n(q + q^{-1})$  have special semisimple quotients $Q_n(q + q^{-1})$ known as the Jones algebras. These algebras may be realized as the endomorphism rings of the images of $V_q^{\otimes n}$ in the socalled fusion category for $U_q(sl_2)$, \cite{A92}, \cite{AS}.  We use this to point out that the dimensions of the simple modules for $Q_n(q + q^{-1})$ are recovered as an especially easy part of our algorithm for the $TL_n(q + q^{-1})$ case. In characteristic $0$  this result was again obtained by Iohara, Lehrer and Zhang \cite{ILZ}. Our result reveals that the answers for $Q_n(q + q^{-1})$ in characteristic $p>0$ are in fact exactly the same.

Consider now instead of $V_q$ another tilting module $T_q$ for $U_q(sl_2)$ and consider the family of endomorphism rings $\End_{U_q(sl_2)}(T_q^{\otimes n})$. The dimensions of the simple modules for these cellular algebras (\cite{AST1}) are as in the case of $V_q$ determined by the decomposition of $T_q^{\otimes n}$ into indecomposable summands. This decomposition can (for instance) be done by formally writing (in the Grothendieck group) $T_q^{\otimes n}$ as a $\Z$-linear combination of $V_q^{\otimes r}$ and then using the decomposition we already have of these latter modules into indecomposables. We obtain as a special case of this the dimensions of simple modules for the $BMW$-algebras.

\section{General $sl_2$-theory}

Let first $q \in k\setminus \{0\}$ be arbitrary and denote by $U_q = U_q(sl_2)$ the quantum group for $sl_2$. To be precise by this we mean the Lusztig version of the quantized enveloping algebra constructed from the generic quantum group via $q$-divided powers by specializing the quantum parameter to $q$. In this section we recall some standard facts from the representation theory for $U_q$. For details see e.g. \cite{AT}.

\subsection{Weyl modules and Weyl filtrations}
The Weyl modules for $U_q$ are denoted $\Delta_q(m), \; m \in \Z_{\geq 0}$. So  $\Delta_q(m)$ has dimension $m+1$. In particular, $\Delta_q(0)$ is the trivial module $k$,  and $V_q =  \Delta_q(1)$ is the natural $2$-dimensional representation of $U_q$. We set $\Delta_q(m) = 0$ if $m < 0$. 

We have for all $m \geq 0$ a short exact sequence 
\begin{equation} \label{basic}
 0 \rightarrow  \Delta_q(m-1)      \rightarrow \Delta_q(m) \otimes V_q \rightarrow \Delta_q(m+1) \rightarrow 0. 
\end{equation}

Recall that a module $M$ is said to have a Weyl filtration if it contains submodules $0=F_0 \subset F_1 \subset \cdots \subset F_r = M$ with $F_i/F_{i-1} \simeq \Delta_q(m_i)$ for some $m_i$'s. We denote then by $(M: \Delta_q(m))$ the multiplicity of $\Delta_q(m)$ as a subquotient in such a filtration. By (\ref{basic})
 we see that $\Delta_q(m) \otimes V$ has a Weyl filtration and that $(\Delta_q(m) \otimes V_q : \Delta_q(r)) = 1$ if $r =m  \pm 1$ and $0$ for all other values of $r$. This implies that  $V_q^{\otimes n}$ has a Weyl filtration for all $n$ and an easy induction gives the wellknown formula, see e.g. \cite{ILZ}.

\begin{equation} \label{Weyl mult}
(V_q^{\otimes n} : \Delta_q(m)) = \binom{n}{r} - \binom{n}{r-1} 
\end{equation}
where $r = (n-m)/2$ and it is understood that $\binom{n}{ r} = 0$ unless $r \in \Z_{\geq 0}$.

\subsection{Tilting modules}

A module $M$ for $U_q$ is called tilting if both $M$ and its dual $M^*$ have Weyl filtrations. As $V_q$ is selfdual we see from (\refeq{basic}) that $V_q^{\otimes n}$ is tilting for all $n$.

For each $m \in \Z_{\geq 0}$ there is a unique indecomposable tilting module $T_q(m) $ with the property that $(T_q(m) : \Delta_q(m)) = 1$ and $(T_q(m) : \Delta_q(t)) = 0$ unless $t \leq m$ and $t \equiv m \; (\mo \; 2)$. Moreover, up to isomorphisms all indecomposable tilting modules are accounted for in this way. Hence if $M $ is a tilting module we have $M = \bigoplus _m T_q(M) ^{(M:T_q(m))}$ for some unique non-negative integers $(M : T_q(m))$. Our aim is to determine the tilting multiplicities in $V_q^{\otimes n}$, i.e. the numbers $(V_q^{\otimes n}:T_q(m))$.

\subsection{The $3$ different cases}

If $q$ is not a root of unity in $k$ then the category of finite dimensional $U_q$-modules is semisimple, see \cite{APW}. This means in particular, that $T_q(m) = \Delta_q(m)$ for all $m$ and that these modules are also the simple modules for $U_q$.  Hence in this case all modules $M$ are tilting and we have
\begin{equation}
(M:T_q(m)) = (M : \Delta_q(m)) \text { for all } m.
\end{equation}
These numbers are the composition factor multiplicities for $M$.
This case is therefore identical to the classical theory for finite dimensional representations of the complex Lie algebra $sl_2(\C)$. 

In this paper we are concerned with the non-semisimple cases. We shall divide our treatment into the following  $3$ cases

1. The characteristic zero root of unity case (Section 3). This case was the main focus of \cite{AT}. See also \cite{ILZ}.

2. The case where $q = 1$ and $k$ has positive characteristic (Section 4). In this case the representation theory for $U_q$ is identical to the modular representation theory for the algebraic group $SL_2$. This case was dealt with in \cite{E}.

3. The positive characteristic root of unity case (Section 5). This is also sometimes called the mixed case (\cite{AW}). See also \cite{ALZ}.

\section{The characteristic zero root of unity case}

In this section we assume that $p = 0$ and that $q$ is a root of unity. We denote by $\ell$ the order of $q^2$. 

\subsection{Weyl factors of indecomposable tilting modules}

We have the following description of $T_q(m)$, see \cite{AT}.

\begin{prop}\label{tiltings}
\begin{enumerate}
\item  If $m < \ell$ then $T_q(m) \simeq \Delta_q(m)$.
\item  If $m \equiv -1 \; (\mo \; \ell)$ then $T_q(m) \simeq \Delta_q(m)$.
\item  If $m = m_1 \ell + m_0$ with $0 \leq m_0 \leq \ell -2$ and $m_1 > 0$ then we have an exact sequence
$$ 0 \to \Delta_q(m) \to T_q(m)     \to \Delta_q(m') \to 0$$
where $ m' = m - 2 m_0 -2$.
\end{enumerate}
\end{prop}

This makes it possible to express the Weyl modules in terms of the indecomposable tilting modules. In fact, let $\mathcal K$ denote the Grothendieck group of the category of finite dimensional $U_q$-modules. Then we write $[M]$ for the class in $\mathcal K$ of a $U_q$-module $M$. Both the set of Weyl module classes $[\Delta_q(m)])$ and the set of classes of indecomposable tilting modules $[T_q(m)])$ constitute bases of the free $\Z$-module $\mathcal K$. By Proposition \ref{tiltings}  the transition between these bases is given by 

\begin{cor}\label{tiltings in K}
\begin{enumerate}
\item  If $m$ satisfies one of the conditions in Proposition \ref{tiltings}(1) and (2) then $[T_q(m)] = [\Delta_q(m)]$.
\item  If $m$ is as in Proposition \ref{tiltings}(3) then \\
$[T_q(m)] = [\Delta_q(m)] + [\Delta_q(m')]$  and  $[\Delta_q(m)] = \sum_{j \geq 0}  [T_q(m-2j\ell)] - \sum
_{j \geq 0}  [T_q(m'-2j\ell)].$
\end{enumerate}
\end{cor}

\subsection{Decomposition of  $V_q^{\otimes n}$ into indecomposable summands}

Set $a_{n,m} = (V_q^{\otimes n} : \Delta_q(m))$. Then the matrix $(a_ {n,m})_{n,m \in \Z_{\geq 0}}$ determines the Weyl module multiplicities in all the tensor powers of $V_q$. Let similarly, $b_{n,m} = (V_q^{\otimes n} : T_q(m))$. Then the decomposition of the tensor powers of $V_q$ into indecomposable tilting modules is given by

\begin{cor}\label{tilting mult}
\begin{enumerate}
\item $b_{n,m} = a_{n,m}$ { if } $m \equiv -1 \; (\mo \; \ell)$.
\item $b_{n,m} = \sum_{j \geq 0} a_{n, m+2j\ell} - \sum_{j \geq 1} a_{n, m'+2j\ell}$  { for all other }$ m$.
\end{enumerate}
(in (2) we have used the notation from Proposition \ref{tiltings} (3)).
\end{cor}

\begin{remarkcounter} \label{warning}
Even though $\Delta_q(m) = T_q(m)$ for certain special values of $m$ it is not necessarily true that for a given tilting module $T$ its Weyl factor multiplicity $(T:\Delta_q(m))$ coincides with the "tilting multiplicity" $(T:T_q(m))$. For instance, the trivial module $k = \Delta_q(0) = T_q(0)$ occurs once as a Weyl factor in $T_q(2\ell-2)$ whereas clearly $(T_q(2\ell - 2): T_q(0)) = 0$. However, if $m \equiv -1 \; (\mo \; \ell)$ then we do have $(T: \Delta_q(m)) = (T:T_q(m))$. 
\end{remarkcounter}

\begin{examplecounter}
In Table  1 we have listed the Weyl module multiplicities in $V_q^{\otimes n}$ for $n=0, 1, \cdots ,16$, i.e. in the $n$'th row we have listed the multiplicities of all Weyl factors in $V_q^{\otimes n}$  (empty spots here and in all other figures mean that the corresponding multiplicities are $0$). This is of course straightforward: we can either use (\refeq{Weyl mult} or we can proceed via induction on $n$ by first observing that $a_{0,j} = \delta_{0,j}$, $a_{i,j} = 0$ for all negative values of $j$,  and then for $i>0$ apply the recurrence relation $a_{i,j} = a_{i-1, j-1} +  a_{i-1, j+1}$. However, it will be convenient to have this table available when computing tilting multiplicities both in the situation of the present section and those coming up. 

\eject
\centerline
{ \it Table  1.  Weyl factor multiplicities in $V^{ \otimes n}$}

\vskip .5cm
\noindent
\begin{tabular}{ r| c  c c c c c c c c c c c c c c c c}
  
   & 0 & 1 & 2 & 3 & 4 & 5 & 6 & 7 & 8 & 9 & 10 & 11 & 12 & 13 & 14 & 15 & 16  \\  \hline 
  0 & 1  &   \\   
  1 &  & 1 \\ 
  2 & 1 & & 1 \\ 
  3 & & 2 &   & 1 \\ 
  4 & 2 & & 3 & & 1 \\
5 & &5 & & 3 & & 1 & \\
6 &5 && 9& &5 & &1  \\
7 & &14 & &14&&6&&1 \\
8 & 14&&28 &&20 &&7 & &1 \\
9 &&42&&48&&27&&8&&1\\
10 &42&&90&&75&&35&&9&&1\\
11&&132&&165&&110&&44&&10&&1\\
12 &132&&297&&275&&154&&54&&11&&1\\
13 &&429&&572&&429&&208&&65&&12&&1\\
14 &429&&1001&&1001&&637&&273&&77&&13&&1\\
15&&1430&&2002&&1638&&910&&350&&90&&14&&1\\
16 &1430&&3432&&3640&&2548&&1260&&440&&104&&15&&1

\end{tabular}
\vskip 1 cm 

Suppose now $\ell = 5$. We can use Table  1 to obtain the tilting multiplicities for the same values on $n$ by applying Corollary \ref{tilting mult}. The  results are listed in Table  2. We have put a vertical line in front of the columns indexed by all $m$ which have residue $-1$ modulo $\ell$. According to Corollary \ref{tilting mult}(1) these columns are identical to the corresponding columns in Table  1.

\vfill \eject

\centerline
{ \it Table  2.  Tilting multiplicities in $V^{ \otimes n}$ for $\ell = 5$}

\vskip .5cm
\noindent
\begin{tabular}{ r| c  c c c |c c c c c| c c c c c |c c c}
  
   & 0 & 1 & 2 & 3 & 4 & 5 & 6 & 7 & 8 & 9 & 10 & 11 & 12 & 13 & 14 & 15 & 16  \\  \hline 
  0 & 1  &&&&&&&&&&&&&&&  \\ 
  1 &  & 1&&&&&&&&&&&&&&& \\ 
  2 & 1 & & 1&&&&&&&&&&&&&& \\ 
  3 & & 2 &   & 1 &&&&&&&&&&&&&\\ 
  4 & 2 & & 3 & & 1&&&&&&&&&& \\
5 & &5 & & 3 & & 1 &&&&&&&&& \\
6 &5 && 8& &5 & &1  &&&&&&&&&\\
7 & &13 & &8&&5&&1&&&&&& \\
8 & 13&&21 &&20 &&7 & &1 &&&&&&\\
9 &&34&&21&&27&&8&&1&&&&&&\\
10 &34&&55&&75&&35&&8&&1&&&&&&\\
11&&89&&55&&110&&43&&10&&1&&&&\\
12 &89&&144&&275&&153&&43&&11&&1&&&\\
13 &&233&&144&&428&&196&&65&&12&&1&&\\
14 &233&&377&&1001&&624&&196&&77&&13&&1&\\
15&&610&&377&&1625&&820&&450&&90&&13&&1&\\
16 &610&&987&&3640&&2445&&820&&440&&103&&15&&1
\end{tabular}
\end{examplecounter}

\vskip 1cm

The following result gives an alternative way of computing tilting multiplicities.

\begin{prop} \label{tilt times V}
Suppose $\ell > 2$. Let $m \in \Z_{\geq 0}$. Then
\begin{enumerate}
\item if $m \equiv -1 \; (\mo \; \ell)$ then $T_q(m) \otimes V_q \simeq T_q(m+1)$,
\item if $m \equiv 0 \; (\mo \; \ell)$ then $T_q(m) \otimes V_q \simeq T_q(m-1)^{\oplus 2} \oplus T_q(m+1)$,
\item if $m \equiv m_0 \; (\mo \; \ell)$ with $0 <m_0 < \ell-2$ then $T_q(m) \otimes V_q \simeq T_q(m-1) \oplus T_q(m+1)$,
\item if $m \equiv \ell - 2 \; (\mo \; \ell)$ then $T_q(m) \otimes V_q \simeq T_q(m+1 - 2\ell) \oplus T_q(m-1) \oplus T_q(m+1) )$.
\end{enumerate}
\end{prop}

\begin{proof}
First observe that tensoring by $V_q$ preserves the tilting property (\refeq{basic}) and recall that the indecomposable tilting modules are uniquely determined by their Weyl factor multiplicities. We have recorded these in Corollary \ref{tiltings in K}  and the result follows now by (\refeq{basic}).
\end{proof}

\begin{cor} \label{q-tilt recurrence}
Let still $\ell > 2$. The tilting multiplicities $b_{n,m}$ of $T_q(m)$ in $V_q^{\otimes n}$ are given by
\begin{enumerate}
\item If $m \equiv -1 \; (\mo \; \ell)$ then $b_{n,m} = a_{n,m} =  \binom{n}{k} - \binom{n}{k-1}$ .
\item If   $m \equiv m_0 \; (\mo \; \ell)$ with $0 \leq m_0 < \ell -2$ then $b_{n,m} = b_{n-1,m-1} + b_{n-1,m+1}$.
\item  If   $m \equiv \ell - 2 \; (\mo \; \ell)$ then $b_{n,m} = b_{n-1,m-1}.$
\end{enumerate}
\end{cor}

\begin{remarkcounter} \label{q-tilt algorithm}
\begin{enumerate}
\item This corollary makes it easy to find the matrix of tilting multiplicities $(b_{n,m})$ inductively (for $\ell > 2$): First we use (1) to fill all columns numbered by an $m$ which has $m \equiv -1 \; (\mo \; \ell)$. Then if $m$ belongs to the interval $[m_1\ell, (m_1+1)\ell -2]$ we get $b_{n,m}$ via (2), respectively (3)  as a sum of two (respectively 1) number(s) from the previous row.
\item It is no coincidence that in Table  2 we can observe that columns 0 and 1, respectively columns 2 and 3,  respectively 7 and 8, respectively 12 and 13 look alike. In fact, we have in general (as it follows from (2) (with $m=0$) and (3))
\begin{equation}
b_{n, 0} = b_{n-1, 1}  \text { and } b_{n, m} = b_{n-1, m-1} \text { for all  } m \equiv -2 \; (\mo \; \ell).
\end{equation}
\item Suppose $\ell = 2$. In this case the analogue of Proposition \ref{tilt times V} has only two cases, namely we have $T(m) \otimes V_q \simeq T(m+1)$ if $m$ is odd (in complete agreement with  Proposition \ref{tilt times V}(1)) whereas if $m$ is even we get  $T_q(m) \otimes V_q \simeq T_q(m+1) \oplus T_q(m-1)^{\oplus 2} \oplus T_q(m-3)$. Hence we deduce that for odd $m$ we have  $b_{n,m} = a_{n,m} =  \binom{n}{r} - \binom{n}{r-1}$ whereas for even $m$ we have  $b_{n,m} = a_{n-1,m-1} =  \binom{n-1}{r} - \binom{n-1}{r-1}$.
\end{enumerate}
\end{remarkcounter}

\begin{examplecounter}
 Suppose $\ell = 3$. Using Remark \ref{q-tilt algorithm} we have found the tilting multiplicities in $V_q^{\otimes n}$ for $n \leq 16$, see Table  3 below. Note that the first two columns contain only $1$'s. This is true for all $n$ because of the identities in Remark \ref{q-tilt algorithm}(2),i.e. the trivial tilting module $T_q(0) = k$ occurs once as a summand of $V_q^{\otimes n}$ for all even $n$'s, and the tilting module $T_q(1) = V_q$ occurs once as a summand of $V_q^{\otimes n}$ for all odd $n$'s. 
\vskip 1cm
\centerline
{ \it Table  3.  Tilting multiplicities in $V^{ \otimes n}$ for $\ell = 3$}

\vskip .5cm
\noindent
\begin{tabular}{ r| c  c |c c c| c c c | c c c| c c c |c c c}
  
   & 0 & 1 & 2 & 3 & 4 & 5 & 6 & 7 & 8 & 9 & 10 & 11 & 12 & 13 & 14 & 15 & 16  \\  \hline 
  0 & 1  &&&&&&&&&&&&&&&  \\ 
  1 &  & 1&&&&&&&&&&&&&&& \\ 
  2 & 1 & & 1&&&&&&&&&&&&&& \\ 
  3 & & 1 &   & 1 &&&&&&&&&&&&&\\ 
  4 & 1 & & 3 & & 1&&&&&&&&&& \\
5 & &1 & & 4 & & 1 &&&&&&&&& \\
6 &1 && 9& &4 & &1  &&&&&&&&&\\
7 & &1 & &13&&6&&1&&&&&& \\
8 & 1&&28 &&13 &&7 & &1 &&&&&&\\
9 &&1&&41&&27&&7&&1&&&&&&\\
10 &1&&90&&41&&34&&9&&1&&&&&&\\
11&&1&&131&&110&&34&&10&&1&&&&\\
12 &1&&297&&131&&144&&54&&10&&1&&&\\
13 &&1&&428&&429&&144&&64&&12&&1&&\\
14 &1&&1001&&428&&573&&273&&64&&13&&1&\\
15&&1&&1429&&1638&&573&&337&&90&&13&&1&\\
16 &1&&3432&&1429&&2211&&1260&&337&&103&&15&&1

\end{tabular}
\end{examplecounter}

\vskip .5cm

\subsection{Fusion}

We keep in this subsection the assumption that $q$ is a root of unity and that $\cha k = 0$.

Let $\mathcal T_q$ denote the category of tilting modules for $U_q$. Inside $\mathcal T_q$ we consider the subcategory $\mathcal N_q$ consisting of all negligible modules, i.e. a module $M \in \mathcal T_q$ belongs to $\mathcal N_q$ iff $\Tr_q(f) = 0$ for all $f \in \End_{U_q}(M)$. As each object in $\mathcal T_q$ is a direct sum of certain of the $T_q(m)$'s and $\dim_q T_q(m) = 0$ iff $m \geq \ell - 1$ we see that $M \in \mathcal N$ iff $(M:T_q(m)) = 0$ for $m= 0, 1, \cdots , \ell -2$.

The fusion category $\mathcal F_q$ is now the quotient category $\mathcal T_q/ \mathcal N_q$. We may think of objects in $\mathcal F_q$ as the tilting modules $Q$ whose indecomposable summands are among the $T_q(m)$'s with $m \leq \ell -2$. Note that $\mathcal F$ is a semisimple category with simple modules $T_q(0), T_q(1), \cdots , T_q(\ell - 2)$.

We proved in \cite{A92} (not just for $sl_2$ but for all semisimple Lie algebras) that $\mathcal N_q$ is a tensor ideal in $\mathcal T_q$. This means that $\mathcal F_q$ is a tensor category. We denote the tensor product in $\mathcal F_q$ by $\underline \otimes$. If $Q_1, Q_2 \in \mathcal F$ then $Q_1 \underline \otimes Q_2 = \pr (Q_1 \otimes Q_2)$ where $\pr$ denotes the projection functor from $\mathcal T_q$ to $\mathcal F_q$ (on the right hand side we consider $Q_1, Q_2$ as modules in $\mathcal T_q$). The following proposition tells us how to work with $\underline \otimes$.
\begin{prop}\label{fusion product}
Let $0 \leq m \leq \ell-2$. Then $T_q(m) \underline \otimes V = \begin{cases} {T_q(m-1) \oplus T_q(m+1) \text { if } m < \ell -2}, \\ {T_q(\ell -3) \text { if } m = \ell -2}. \end{cases}$
\end{prop}

\begin{proof} Recall that for all $m$ in question we have $T_q(m) = \Delta_q(m)$ and use (\refeq{basic}). Alternatively, this a special case of Proposition \ref{tilt times V}.
\end{proof}

When $0 \leq m \leq \ell -2$ and $n \in \Z_{\geq 0}$ we denote by $\underline b_{n,m}$ the tilting multiplicity of $T_q(m)$ in $V_q^{\underline \otimes n}$. Note that this is also the tilting multiplicity of  $T_q(m)$ in $V_q^{\otimes n}$ (for our range of $m$'s) , i.e. the matrix $(\underline b_{n,m})_{n \geq 0, 0 \leq m \leq \ell -2}$ is the submatrix of the matrix $(b_{n,m})_{n,m}$ in Section {3.2} consisting of the first $\ell -1$ columns. If we set $\underline b_{n, \ell -1} = 0 = \underline b_{n,-1}$  for all $n$ then Proposition \ref{fusion product} tells us that we can determine these multiplicities by

\begin{cor} \label{fusion numbers}
We have $\underline b_{0,m} = \delta_{0,m}$ and  $\underline b_{n,m} = \underline b_{n-1, m-1} + \underline b_{n-1, m+1}$ for $n>0$ and $0 \leq m \leq \ell -2$.
\end{cor}

\section{Positive characteristic with $q=1$}

Now we consider the case where $\cha k = p>0$ and $q = 1 \in k$. As remarked above the representation theory of $U_q$ is identical to the modular representation theory of the algebraic group $SL_2$. We shall use the same notation as in the previous section  except that we shall drop the index $q$ in our notation for Weyl modules and indecomposable tilting modules. In particular, $V$ will denote the natural $2$-dimensional module for $SL_2$. Moreover, $p$ will now play the same role as $\ell$ did in Section 3.

In this case K. Erdmann \cite{E} worked out the behavior of indecomposable tilting modules and their multiplicities in the tensor powers of $V$. Part of our treatment below overlaps with her paper, to which we refer for further details. In particular, Erdmann obtains explicit formulae for the generating functions $\sum_{n \geq 0} (V^{\otimes n} : T(m)) z^n$.

\subsection{Weyl multiplicities}

We still have for all $m \geq 0$ (again setting $\Delta(-1) = 0$ and now  $V = \Delta(1)$ is the natural $2$-dimensional module for $SL_2$) the short exact sequence
\begin{equation}
0 \to \Delta(m+1) \to \Delta(m) \otimes V  \to \Delta(m-1) \to 0. 
\end{equation}
This means in particular that the Weyl multiplicities in $V^{\otimes n}$ are exactly as before, i.e. given by (\refeq{Weyl mult}).

\subsection{The first few indecomposable tilting modules}
The behavior of indecomposable tilting modules for $SL_2$ begins as for $U_q$ when $q$ is a $p$'th root of unity. In fact, an easy direct calculation gives

\begin{lem} \label{small m}
Suppose $m < p^2 + p - 1$. Then 
\begin{enumerate}
\item $T(m) = \Delta(m)$ when $m<p$ as well as when $m \equiv -1 \; (\mo \; p) $.
\item If $m= m_1 p + m_0$ with $0 \leq m_0 < p$ and $m_1 > 0$ then we have a short exact sequence
$$ 0 \to \Delta(m) \to T(m) \to \Delta(m') \to 0$$
where $m' = m - 2m_0 -2$.
\end{enumerate}
\end{lem}

\subsection{Donkin's tensor product theorem}

To obtain the Weyl multiplicities in $T(m)$ for larger $m$'s we shall employ Donkin's tensor product theorem for indecomposable tilting modules, \cite{Do}. Donkin has proved this result for all semisimple algebraic groups when $p \geq 2h-2$. Note that in our case this means that it is known for all $p$. To formulate it we need the Frobenius endomorphism $F$ on $SL_2$. This is the map which raises the entries of a matrix in $SL_2$ to their $p$'th powers. If $M$ is a module for $SL_2$ we denote by $M^{(1)}$ its Frobenius twist, i.e. the same vector space but with the action precomposed by $F$. When we iterate $F$ we obtain the higher Frobenius twists $M^{(r)}$, $r \geq 0$. In this notation we have 

\begin{prop} \cite{Do} \label{Donkin}

Let $m, r \in \Z_{>0}$ and assume $m \geq p^r -1$. Write $m = \tilde m_1 p^r + \tilde m_0$ with $p^r -1 \leq \tilde m_0 \leq 2p^r -2$ . Then 
$$ T(m) \simeq T(\tilde m_1)^{(r)} \otimes T(\tilde m_0).$$

\end{prop}

\begin{remarkcounter} \label{proof Donkin}
It is easy to check this result directly in our $SL_2$ case: Denote by $St_r$ the $r$'th Steinberg module. This is the simple module with highest weight $p^r - 1$. By the linkage principle $St_r = \Delta (p^r-1)$, cf. \cite{A80a}. Therefore $St_r$ is tilting. Now $T(\tilde m_0)$ is a summand of $St_r \otimes V^{\otimes \tilde m_0 -(p^r-1)}$ and therefore the right hand side in the proposition is a summand of $T(\tilde m_1)^{(r)} \otimes St_r \otimes  V^{\otimes 
\tilde m_0 - (p^r - 1)}$. It is wellknown (and easy to check directly in this case) that $\Delta(s)^{(r)} \otimes St_r \simeq \Delta(sp^r + p^r -1)$ for all $s$. It follows that the right hand side is tilting. To see that it is indecomposable one verifies that it has simple socle. Having the same highest weight as the left hand side we get the isomorphism.
\end{remarkcounter}

\begin{remarkcounter} \label{2p-1}
As a special case of Proposition \ref{Donkin}  we have  $T(2p^r - 1) \simeq T(1)^{(r)} \otimes St_r = V^{(r)} \otimes St_r$.
\end{remarkcounter}

\subsection {The case $p=2$}

Consider now the special case $p=2$. The first few tilting modules are easy to find (e.g. by using Lemma \ref{small m})
$$ T(0) = k,\; T(1) = V,\; T(2) = V^{\otimes 2}$$
and then we can use Proposition \ref{Donkin} to find the rest. For instance $T(3) = T(1)^{(1)} \otimes T(1) = V^{(1)} \otimes V$ and $T(4) = T(1)^{(1)} \otimes T(2) = V^{(1)} \otimes V \otimes V$.  We also record the result of tensoring these modules with $V$:
$$ T(0) \otimes V = T(1),\;T(1) \otimes V = T(2), \;T(2) \otimes V = T(3) \oplus T(1)^{\oplus 2}, \; T(3) \otimes V = T(4).$$

The following proposition records what happens in general when we tensor an indecomposable tilting module by $V$.

\begin{prop} \label{p=2}
Let $p=2$. Then for $m \in \Z_{\geq 0}$ we have
$$ T(m) \otimes V = \begin{cases} {T(m+1) \text { if } m \text { is odd },} \\ {T(m+1) \oplus (\bigoplus_{s=1}^r T(m+1 - 2^s))^{\oplus 2} } \text { if $m$ is even}\end{cases},$$
where $r = r(m)$ is the largest integer for which $m \equiv - 2 \; (\mo \; 2^{r
})$.
\end{prop}

\begin{proof} Note that the formulas above proves the proposition for $m \leq 3$.

Consider now the case where $m$ is odd and write $m = 2m_1 + 1$. Then by Proposition \ref{Donkin} we have $T(m) = T(m_1)^{(1)} \otimes V$ so that $T(m) \otimes V = T(m_1)^{(1)} \otimes T(1) \otimes V = T(m_1)^{(1)
} \otimes T(2 )  = T(m+1)$.

Next suppose $m$ is even and write with $r$ as in the proposition $m = 2^{r} -2 + m_1 2^{r+1} = 2 + 2^2 + \cdots + 2^{r-1} + m_1 2^{r+1}$. Then we use Proposition \ref{Donkin} once again to see that $T(m) = T(2^{r-1} -2 + m_1 2^{r})^{(1)} \otimes T(2)$. By induction on $m$ combined with the above formula for tensoring $T(2)$ with $V$ we then get $T(m) \otimes V =  T(2^{r-1} -2 + m_1 2^{r})^{(1)} \otimes (T(2) \otimes V) =  T(2^{r-1} -2 + m_1 2^{r})^{(1)} \otimes (T(3) \oplus T(1)^{\oplus 2}) =
 (T(2^{r-1} -2 + m_1 2^{r}) \otimes V)^{(1)} \otimes V \oplus  T(2^{r-1} -2 + m_1 2^{r})^{(1)} \otimes T(1)^{\oplus 2} =  (T(2^{r-1} -1 + m_1 2^{r})^{(1)} \oplus (\bigoplus_{s=1}^{r-1}   T(2^{r-1} -2 + m_1 2^{r} - 2^s)^{(1)})^{\oplus 2} \otimes V \oplus T(m-1)^{\oplus 2} = T(m+1) \oplus  (\bigoplus_{s=1}^{r-1}   T(m+1 -2^{s+1}))^{\oplus 2}  \oplus T(m-1)^{\oplus 2}$, which is the desired formula.
\end{proof}

Note that for all $t \in 2\Z_{\geq 0}$ and $s \in \Z_{>0}$ we have

\begin{equation}
r(t + 2^s) = \begin{cases} {r(t) \text { if } s > r(t) } \\ {s \text { if } s < r(t)} \end{cases},
\end{equation}
and if $s = r(t)$ we have $r(t + 2^s) > r(t)$.

Using this we get from Proposition \ref{p=2}

\begin{cor} \label{recurrence p=2}
Let $p = 2$. Then for all $m \in \Z_{\geq 0}$ we have $(V^{\otimes 0}: T(m)) = \delta_{0,m}$, and for $n>0$
$$ 
(V^{\otimes n} : T(m)) = \begin{cases} { (V^{\otimes n-1}:T(m-1)) \text { if } m \text { is even,}}\\ {(V^{\otimes n-1}:T(m-1)) + 2 \sum_{s=1}^{r(m-1)} (V^{\otimes n-1}:T(m-1+2^s))} \text { if $m$ is odd.} \end{cases} 
$$
\end{cor}

\begin{examplecounter}
This corollary allows us to determine the tilting multiplicities in $V^{\otimes n}$ inductively. In Table  4 we have listed the results for $n \leq 20$. Note that we have only given the results for $n$ odd since by the corollary we have for $n$ even that  $(V^{\otimes n}:T(m)) = (V^{\otimes n-1}:T(m-1))$. Note also that $(V^{\otimes n}:T(m)) = 0$ unless $n$ and $m$ have the same parity. Therefore the figure also has columns indexed by only odd numbers.

\vfill \eject
\centerline
{ \it Table  4.  Tilting multiplicities in $V^{ \otimes n}$ for $p= 2$}

\vskip .5cm
\noindent
\begin{tabular}{ r| c  c c c c c c c c c c c c c c c c}
  
   & 1 & 3 & 5 & 7 & 9 & 11 & 13 & 15 & 17& 19 &   \\  \hline 
  
  1 & 1 &  \\ 
  3 & 2 &  1 \\ 
  5 & 4& 4 & 1   \\
  7 & 8 &14 & 6 & 1\\
9&16 &48 & 26& 8 & 1 & \\
11 &32 &164& 100&44 &10  &1  \\
13&64 &560 &364 &208&64&12&1 \\
15& 128&1912&1288 &910&336 &90&14 & 1& \\
17&256&6528&4488&3808&1582&544&118&16&1\\
19 &512&22288&15504&15504&6972&2906&780&152&18&1\\

\end{tabular}

\end{examplecounter}

\begin{remarkcounter}
As the figure suggests we have 
\begin{enumerate}
\item  If $n = 2n_1+1$ then $(V_q^{\otimes n}:T(1))= 2^{n_1}$. In fact, Corollary \ref{recurrence p=2} gives first $(V^{\otimes n}:T(0)) = 0 $ for all $n>0$, and then $ (V^{\otimes n}:T(1)) = 2 ((V^{\otimes n-1}:T(2)) = 2(V^{\otimes n-2}:T(1))$. Equivalently, we have for $n$ positive and even, say $n = 2n_1$ that
$(V^{\otimes n}:T(2)) = 2^{n_1-1}$.

\item At the other extreme we have $(V^{\otimes n}:T(n-2)) = n-1$ for all $n$. In fact, if $n$ is odd then Corollary \ref{recurrence p=2} gives $(V^{\otimes n}:T(n-2)) =(V^{\otimes n-1}:T(n-2)) + 2(V^{\otimes n-1}:T(n-1)) = 
(V^{\otimes n-2}:T(n-3)) + 2$.
\end{enumerate}
\end{remarkcounter}

\subsection{Tilting multiplicities for $p>2$}
Now we shall assume $p>2$.  Again in this case we are going to find the tilting multiplicities of $V^{\otimes n}$ by induction on $n$. Therefore we need to determine $T(m) \otimes V$. The first results towards this is

\begin{lem}
\begin{enumerate}
\item $T(p-1) \otimes V \simeq T(p)$,
\item $T(m) \otimes V \simeq T(m+1) \oplus T(m-1)$ if $m \leq 2p-2$ and $m \neq p-1$.
\end{enumerate}
\end{lem}

\begin{proof} Immediate from Lemma \ref{small m}.
\end{proof}

\begin{prop} \label{not-2}
Let $m \geq p-1$ and write $ m = \tilde m_1 p + \tilde m_0$ with $p-1 \leq \tilde m_0 \leq 2p-2$. Assume $\tilde m_0 < 2p-2$. Then we have 
$$T(m) \otimes V \simeq \begin{cases} {T(m+1) \text { if } m \equiv -1 \; (\mo \;  p),} \\ {  T(m+1) \oplus T(m-1)^{\oplus 2} \text { if } m \equiv 0 \; (\mo \; p) }\\ { T(m+1) \oplus T(m-1) \text { otherwise.} } \end{cases} $$

\end{prop}

\begin{proof} Combine Lemma \ref{small m} and Proposition \ref{Donkin}.
\end{proof}

These results give the following recurrence relations.
\begin{cor} \label{not -1}
Let $m,n \in \Z_{\geq 0}$ and assume $m$ is not equivalent to $ -1$ modulo $p$. Then 
$$(V^{\otimes n}:T(m)) = \begin{cases} {(V^{\otimes n-1}:T(m-1)) \text { if } m \equiv -2 \; (\mo \; p)} \\ { (V^{\otimes n-1}:T(m-1)) + (V^{\otimes n-1}:T(m+1)), \text { otherwise.} } \end{cases}$$
\end{cor}

Note that Proposition \ref{not-2} does not include any statement about the case when $m \equiv -2 \; (\mo \; p)$. The reason is that if $\tilde m_0 = 2p-2$ then $T(\tilde m_0) \otimes V$ contains  $T(2p-1)$. Here $2p-1$ is "out of range" with respect to Proposition \ref{Donkin}. This means that in  Corollary \ref{not -1} we have to exclude the case where $m \equiv -1 \;  (\mo \; p)$.  We have two ways of dealing with this remaining case. The first is a continuation of the above arguments where as we shall see things get a bit more elaborate. In Section 4.7 below  we give an alternative way of handling this case. 

So assume now that $m \equiv -2 \; (\mo \;p)$. We start out with the case where $m = p^r-2$ or  $2p^r-2$.

\begin{lem} \label {1 og 2}
Let $m \in \{ p^r -2, 2p^r -2\}$ with $r>0$. Then
$$T(m) \otimes V  \simeq T(m+1) \oplus (\bigoplus_{s=0}^{r-1} T(m+1-2p^s)$$.
\end{lem}

\begin{proof}
We use induction on $r$. If $r = 1$ the statement follows from Lemma \ref{small m}. So suppose $r>1$. By Proposition \ref{Donkin} we have $T(p^r-2) = T(p^{r-1} -2)^{(1)}\otimes T(2p-2)$. Hence using the case $r=1$ combined with Remark \ref{2p-1} and again Proposition \ref{Donkin} we get
$$ T(m) \otimes V \simeq T(p^{r-1}-2)^{(1)} \otimes (V^{(1)} \otimes St_1 \oplus T(2p-3)) \simeq (T(p^{(r-1)}-2) \otimes V)^{(1)} \otimes St_1 \oplus T(p^r-3).$$ 
Induction and one more appeal to Proposition \ref{Donkin} now finish the proof in this case. The case $m = 2p^r-2$ is completely similar.
\end{proof}

\begin{prop} \label{=2p-2}
Let $m \geq p-1$ with $m \equiv -2 \; (\mo \; p)$. Choose $r$ maximal with $m \equiv -2 \; (\mo \; p^r)$ and write $ m = m_1p^r + 2p^r-2$. Assume $m_1 >0$. Then we have 
$$ T(m) \otimes V \simeq T(m+1) \oplus (\bigoplus_{s=0}^{r-1} T(m+1 - 2p^s)) \oplus  \begin{cases} { 0 \text { if } m_1 \equiv -1 \;  (\mo \; p),} \\  {T(m+1 -2p^r)^{\oplus 2} \text { if }  m_1 \equiv 0 \; (\mo \; p),} \\ {T(m+1 -2p^r) \text { otherwise. } }\end{cases}  $$

\end{prop}

\begin{proof}
By Lemma \ref{1 og 2} (and using the same arguments as in its proof) we get $T(m) \otimes V \simeq T(m_1)^{(r)} \otimes T(2 p^r-2) \otimes V \simeq T(m_1)^{(r)} \otimes (T(2 p^r-1 ) \oplus (\bigoplus_{s=0}^{r-1} T(2p^r-1 -2p^s) ) \simeq  T(m_1)^{(r)} \otimes V^{(r)} \otimes St_r \oplus (\bigoplus_{s=0}^{r-1} T(m+1 -2 p^s)) $. Note that by our choice of $r$ we cannot have $m_1 \equiv -2 \; (\mo \; p)$. Hence 
Proposition \ref{not-2} gives us the decomposition of  $T(m_1) \otimes V$.  When we insert this above the formula falls out.
\end{proof}

Note that the statement in this proposition is still valid for the $m$'s dealt with in Lemma \ref{1 og 2} (when $m= p^r-2$  or $m = 2p^r - 2$ the last term in the proposition vanishes because of our convention that $T(t) = 0$ for $t < 0$). 

These results lead to the following recurrence relation for the tilting multiplicities $(V^{\otimes n} : T(m))$ in the case when $m \equiv -1 \; (\mo \;p)$.

\begin{cor} \label{-1}
Let $n, m \in \Z_{>0}$ and suppose $m = b p^j -1$ for some $b$ not divisible by $p$ and $j>0$. Then
$$ (V^{\otimes n}:T(m)) = (V^{\otimes n-1}:T(m-1)) + 2 \sum_{s=0}^{j-1}  (V^{\otimes n-1}:T(m-1 + 2p^s)) + \begin{cases} {0 \text { if } b \equiv -1 \; (\mo \; p) } \\ {(V^{\otimes n-1}:T(m-1 + 2p^j)) \text { otherwise.} }\end{cases} $$

\end{cor}

\subsection{The tilting multiplicity algorithm}

As a first step we determine the tilting multiplicities $(V^{\otimes n} : T(m))$ for $m < p-1$. These are given by the recurrence relations

\begin{prop} Suppose $m < p-1$. Then for $n \geq 0$ we have
$$ (V^{\otimes n}:T(m)) = \begin{cases} { \delta_{0,m} \text { if } n = 0,} \\ {(V^{\otimes n-1}:T(m-1)) +  (V^{\otimes n-1}:T(m+1)) \text { if } n > 0 \text { and } m< p-2} \\ { (V^{\otimes n} : T(p-3)) \text { if } m=p-2.} \end{cases}$$
\end{prop}

\begin{proof}
This follows directly from Corollary \ref{not -1}.
\end{proof}

\begin{remarkcounter}
Note that this proposition says that for $m<p-1$ the tilting multiplicities $(V^{\otimes n} :T(m))$ are the same as when $q$ is a root of unity of order $p$  in a characteristic $0$ field, see the recurrence relations for $b_{n,m}$ in Corollary \ref{q-tilt recurrence}. An alternative argument for this (as well as an alternative way of finding these tilting multiplicities) comes from the formula
$$ (T: T(m)) = \sum _a (T:\Delta(m + 2ap)) + \sum_{b > 0} (T: \Delta (-m-2 + 2bp)),$$
which is valid for all tilting modules $T$ and all $m < p-1$, see \cite{AP}. Note that if $T = V^{\otimes n} $ then the numbers on the right hand side are given by (\refeq{Weyl mult}).
\end{remarkcounter}

To find the multiplicities $(V^{\otimes n} : T(m))$ for $m \geq p-1$ we first consider the case where $m \equiv -1 \; \; (\mo \; p)$, i.e. when $m = p\cdot s$ for some $s \geq 0$. In this case $(V^{\otimes n} : T(p\cdot s))$ is given by Corollary \ref{-1}. Once these numbers are determined we find $(V^{\otimes n} : T(m))$ for $p \cdot s < m < p\cdot (s+1)$ by using Corollary \ref{not -1}.

\begin{examplecounter}
We have illustrated our algorithm in Table  5 by listing the tilting multiplicities $(V^{\otimes n} : T(m))$  in the case where $p=3$ and $n \leq 16$. Note that (as we see in the figure and in general from Corollary \ref{not -1} ) for all $m \equiv -2 \;  (\mo \; p)$ we have  $(V^{\otimes n} : T(m)) =(V^{\otimes n-1} : T(m-1))$. Also note that because of Lemma \ref{small m} the first 10 rows are identical to those in Table  3 (for a general prime $p$ we have by this lemma $(V^{\otimes n} : T(m)) = (V_q^{\otimes n}: T_q(m)$ for all $n < p^2 +p-1$ where the right hand side of the equality refers to the situation from Section 3  with $q$ being a root of unity in a characteristic $0$ field and $\ell = p$). Note that this bound is sharp: $(V^{\otimes 11}: T(5)) = 109 \neq 110 = (V_q^{\otimes 11}:T_q(5))$ (and in general $(V^{\otimes p\cdot p} : T(p\cdot (p-2))) = a_{p\cdot p, p\cdot(p-2)} -1 < a_{p\cdot p, p\cdot(p-2)} = (V_q^{\otimes p\cdot p}: T_q(p\cdot(p-2))$, cf. Proposition 4.12 below).

\vfill \eject
\centerline
{ \it Table  5.  Tilting multiplicities in $V^{ \otimes n}$ for $p = 3$}

\vskip .5 cm
\noindent
\begin{tabular}{ r| c  c |c c c| c c c | c c c| c c c |c c c}
  
   & 0 & 1 & 2 & 3 & 4 & 5 & 6 & 7 & 8 & 9 & 10 & 11 & 12 & 13 & 14 & 15 & 16  \\  \hline 
  0 & 1  &&&&&&&&&&&&&&&  \\ 
  1 &  & 1&&&&&&&&&&&&&&& \\ 
  2 & 1 & & 1&&&&&&&&&&&&&& \\ 
  3 & & 1 &   & 1 &&&&&&&&&&&&&\\ 
  4 & 1 & & 3 & & 1&&&&&&&&&& \\
5 & &1 & & 4 & & 1 &&&&&&&&& \\
6 &1 && 9& &4 & &1  &&&&&&&&&\\
7 & &1 & &13&&6&&1&&&&&& \\
8 & 1&&28 &&13 &&7 & &1 &&&&&&\\
9 &&1&&41&&27&&7&&1&&&&&&\\
10 &1&&90&&41&&34&&9&&1&&&&&&\\
11&&1&&131&&109&&34&&10&&1&&&&\\
12 &1&&297&&131&&143&&54&&10&&1&&&\\
13 &&1&&428&&417&&143&&64&&12&&1&&\\
14 &1&&1000&&428&&560&&273&&64&&13&&1&\\
15&&1&&1428&&1548&&560&&337&&90&&13&&1&\\
16 &1&&3417&&1428&&2108&&1260&&337&&103&&15&&1

\end{tabular}

\vskip 1cm
\end{examplecounter}

\subsection{Steinberg class multiplicities in $V^{\otimes n}$}

Proposition \ref{small m} gives us the Weyl factor multiplicities of the first few $T(m)$'s and then Proposition \ref{Donkin} tells us how to obtain the same information for the remaining $T(m)$'s as long as $m$ is not congruent to $-1$ modulo $p$. In this subsection we give a procedure for calculating the Weyl factors in $T(m)$ as well as the tilting multiplicities $(V^{\otimes n}:T(m))$ when $m \equiv -1 \; (\mo \; p)$.

Recall the dot-action of $p$ on $\Z$:  $p \cdot r = p(r+1) - 1$. Then $p \cdot \Z_{\geq 0}$ is the set of non-negative integers congruent to $-1$ modulo $p$. If $M$ is an $SL_2$-module then its Steinberg class component is the largest submodule in $M$ all of whose composition factors have highest weights in $p\cdot \Z_{\geq 0}$. It is a summand of $M$ by the linkage principle, \cite{A80a}, and it equals $\Hom_{G_1}(St, M) \otimes St$, see \cite{A17}. Here $G_1$ is the (scheme theoretic) kernel of the Frobenius homomorphism $F$. In particular, we have for all $m$

\begin{equation} \label{Steinberg Weyl modules}
\Delta(p\cdot m) \simeq  \Delta(m)^{(1)} \otimes St,
\end{equation}
and
\begin{equation} \label{Steinberg tilting modules}
T(p\cdot m) \simeq  T(m)^{(1)} \otimes St.
\end{equation}
Here the first isomorphism is the Andersen-Haboush isomorphism, see e.g. \cite{A80b}, and the second one is a special case of Proposition \ref{Donkin}.

Recall that the Weyl multiplicities $a_{n,m} = (V^{\otimes n}:\Delta(m))$ are given by the formula in \refeq{Weyl mult}. Let $c_{r,s} = (T(r): \Delta(s))$ denote the Weyl factor multiplicity of $\Delta(s)$ in $T(r)$. Then the matrix $(c_{r,s})_{r,s \geq 0} $ is lower triangular with $1$'s on the diagonal (and by the linkage principle $c_{r,s} = 0$ unless $s \in (r-2p\Z) \cup (r' - 2p\Z)$ where $r= r_1p + r_0$ with $0 \leq r_0 \leq p-1$ and $r' = r_1p -r_0 - 2$). We denote the inverse matrix $(d_{i,j})_{i,j \geq 0}$. This matrix is again lower triangular with $1$'s on the diagonal and it has integer coefficients (not all non-negative).

By the linkage principle we get that $c_{r, p\cdot s} = 0$ unless $r \in p\cdot s +2p\Z_{\geq 0}$ (thus in particular unless $ r \in p \cdot \Z$). Moreover by \refeq{Steinberg Weyl modules} and \refeq{Steinberg tilting modules} we have $c_{p\cdot r, p\cdot s} = c_{r,s}$ for all $r,s$. Of course these properties are inherited by the $d_{i,j}$'s.

Putting these facts together we obtain the following formula for the ''Steinberg class" multiplicities in $V^{\otimes n}$.
\begin{prop} \label{Steinberg mult}
Let $n,s \in \Z_{\geq 0}$. Then $(V^{\otimes n}:T(p\cdot s)) = \sum _r  a_{n,p\cdot r} d_{r,s}$ where the sum runs over those $r$ which satisfy
 $r \geq s$, $p\cdot r \leq n$ and $s \in (r + 2p\Z) \cup (r' + 2p\Z)$.
\end{prop}

\subsection{Fusion}

Like in the situation we dealt with in Section 3 the category of tilting modules for $SL_2$ also  has a quotient called the fusion category.  The construction is similar:

Let $\mathcal T$ denote the category of tilting modules for $SL_2$. Inside $\mathcal T$ we consider the subcategory $\mathcal N$ consisting of all negligible modules, i.e. a module $M \in \mathcal T$ belongs to $\mathcal N$ iff $\Tr(f) = 0$ for all $f \in \End_{SL_2}(M)$. As each object in $\mathcal T$ is a direct sum of certain of the $T(m)$'s and $\dim T(m) $ is divisible by $p$ iff $m \geq p - 1$ we see that $M \in \mathcal N$ iff $(M:T(m)) = 0$ for $m= 0, 1, \cdots , p -2$.

The fusion category $\mathcal F$ is the quotient category $\mathcal T/ \mathcal N$. We may think of objects in $\mathcal F$ as the tilting modules $Q$ whose indecomposable summands are among the $T(m)$'s with $m \leq p -2$. Note that $\mathcal F$ is a semisimple category with simple modules $T(0), T(1), \cdots , T(p - 2)$. 

If $p=2$ this means that $\mathcal F$ is the category with one simple object $T(0) = k$. In the rest of this subsection we therefore assume $p>2$.

As in the quantum case we get that  $\mathcal N$ is a tensor ideal in $\mathcal T$ so that $\mathcal F$ becomes a tensor category. We also denote the tensor product in $\mathcal F$ by $\underline \otimes$. If $Q_1, Q_2 \in \mathcal F$ then $Q_1 \underline \otimes Q_2 = \pr (Q_1 \otimes Q_2)$ where $\pr$ denotes the projection functor from $\mathcal T$ to $\mathcal F$ (on the right hand side we consider $Q_1, Q_2$ as modules in $\mathcal T$). The following proposition tells us how to work with $\underline \otimes$. It is proved completely as in the quantum case.
\begin{prop}\label{small tilt times V}
Let $0 \leq m \leq p-2$. Then $T(m) \underline \otimes V = \begin{cases} {T(m-1) \oplus T(m+1) \text { if } m < p -2}, \\ {T(p -3) \text { if } m = \ell -2}. \end{cases}$
\end{prop}

This means that the tilting multiplicities in $V^{\underline \otimes n}$ are the same as in the quantum case:

\begin{cor}. Let $q$ be a root of unity in a field of characteristic $0$ and suppose its order is $p$. Then we have for all $n \in \Z_{\geq 0}$ and all $m \in [0, p-2]$
$$ (V^{\underline \otimes n}: T(m)) =  (V_q^{\underline \otimes n}: T_q(m)). $$
\end{cor}

\begin{remarkcounter}
Note that $(V^{\underline \otimes n}: T(m)) =  (V^{\otimes n}: T(m)) $ for all $m \leq p-2$. Hence the numbers  $ (V^{\underline \otimes n}: T(m))_{n \geq 0, 0 \leq m \leq p-2}$ are the numbers in the first $p-1$ columns of the tilting multiplicity matrix $ (V^ {\otimes n}: T(m))_{n,m \geq 0}$ we have studied in this section. They satisfy the same recursion rule as the corresponding quantum numbers, cf. Corollary \ref{fusion numbers}.
\end{remarkcounter}

\section{Positive characteristic with $q$ a non-trivial root of $1$}

In this section $k$ will still be a field of characteristic $p>0$. We shall here consider a root of unity $q \in k\setminus \{\pm 1\}$. As in Section 3 we denote by $\ell$ the order of $q^2$. We shall use the same notation as in the previous sections except now we add an upper index $p$ to the notation for Weyl modules and indecomposable tilting modules, i.e. we write $\Delta_q^p(m)$, respectively $T_q^p(m)$, for the Weyl module, respectively the indecomposable tilting module with highest weight $m$. We write $V_q^p $ for the two dimensional Weyl module $\Delta_q^p(1)$.

Many of the arguments in this case resembles those from the previous cases. Therefore we only give details when things are different.

\subsection{Weyl multiplicities of $(V_q^p)^{\otimes n}$}
We still have a short exact sequence
$$ 0 \to \Delta_q^p(m+1) \to \Delta_q^p(m) \otimes V_q^p \to \Delta_q^p(m-1) \to 0$$
for all $m \in \Z_{\geq 0}$. This means that formula (\refeq{Weyl mult}) remains valid for the Weyl factors multiplicities in $(V_q^p)^{\otimes n}$.

\subsection{The first few indecomposable tilting modules}
Our story on indecomposable tilting modules begins just like it did in section 4 (now replacing $p$ by $\ell$).

\begin{lem} \label{small m mixed}
Suppose $m < p \ell + \ell - 1$. Then 
\begin{enumerate}
\item $T_q^p(m) = \Delta_q^p(m)$ when $m<\ell $ as well as when $m \equiv -1 \; (\mo \; \ell) $.
\item If $m= m_1 \ell + m_0$ with $0 \leq m_0 < \ell$ and $m_1 > 0$ then we have a short exact sequence
$$ 0 \to \Delta_q^p(m) \to T_q^p(m) \to \Delta_q^p(m') \to 0$$
where $m' = m - 2m_0 -2$.
\end{enumerate}
\end{lem}

\subsection{A tensor product theorem a la Donkin's}

Recall that we have a quantum Frobenius homomorphism $F_q : U_q \to U_1$, see e.g. \cite{L} or \cite{AW}. This allows us for each $U_1$-module $M$ (i.e. each module for $SL_2$ over $k$) to consider the $U_q$-module $M^{[q]}$, namely the module obtained from $M$ by precomposing the action by $F_q$. We then have the following analogue of Donkin's tensor product theorem, see Proposition \ref{Donkin}.

\begin{prop} \label{q-Donkin}

Let $m \in \Z$ and assume $m \geq \ell -1$. Write $m = \tilde m_1 \ell+ \tilde m_0$ with $\ell -1 \leq \tilde m_0 \leq 2\ell -2$. Then 
$$ T_q^p(m) \simeq T(\tilde m_1)^{[q]} \otimes T_q^p(\tilde m_0).$$
\end{prop}
\begin{proof}
Imitate the arguments from Remark \ref{proof Donkin} (use now the quantum Steinberg module $St_q = \Delta_q^p(\ell -1)$).
\end{proof}

Note that Lemma \ref{small m mixed} determines (in particular) the indecomposable tilting modules $T_q^p(m)$ for all $m \leq 2 \ell -2$. Proposition \ref{q-Donkin} then gives us all those with $m > 2 \ell -2$ when we as input use the information from Section 4 on the modules $T(r)$, $r>0$.

\subsection{Tilting multiplicities}

In this section we shall deduce an algorithm in our mixed case. We shall proceed as in Section 4. The case where $\ell = 2$ needs special attention. 
\begin{lem} \label{small tilt times V}
Let $m \leq 2 \ell -2$. Then
$$T_q^p(m) \otimes V_q^p \simeq  \begin{cases} {T_q^p(\ell) \text { if } m= \ell -1,}\\ {T_q^p(\ell+1)  \oplus T_q^p(\ell-1)^{\oplus 2} \text { if } m = \ell, } \\ {T_q^p(m+1)  \oplus T_q^p(m-1) \text { otherwise.}} \end{cases}$$
\end{lem}

\begin{proof} Immediate from Lemma \ref{small m mixed}
\end{proof}

\begin{prop} \label{q-not-2}
Let $m \geq \ell-1$ and write $ m = \tilde m_1 \ell + \tilde m_0$ with $\ell-1 \leq \tilde m_0 \leq 2\ell-2$. Assume $\tilde m_0 < 2\ell-2$. Then we have 
$$T_q^p(m) \otimes V_q^p \simeq \begin{cases} {T_q^p(m+1) \text { if } m \equiv -1 \; (\mo \;  \ell),} \\ {  T_q^p(m+1) \oplus T_q^p(m-1)^{\oplus 2} \text { if } m \equiv 0 \; (\mo \; \ell) }\\ { T_q^p(m+1) \oplus T_q^p(m-1) \text { otherwise.} } \end{cases} $$

\end{prop}

\begin{proof} Combine Lemma \ref{small tilt times V} and Proposition \ref{q-Donkin}.
\end{proof}

These results give the following recurrence relations.
\begin{cor} \label{q-not-1}
Let $m,n \in \Z_{\geq 0}$ and assume $m$ is not equivalent to $ -1$ modulo $\ell$. Then 
$$((V_q^p)^{\otimes n}:T_q^p(m)) = \begin{cases} {((V_q^p)^{\otimes n-1}:T_q^p(m-1)) \text { if } m \equiv -2 \; (\mo \; \ell),} \\ { ((V_q^p)^{\otimes n-1}:T_q^p(m-1)) + ((V_q^p)^{\otimes n-1}:T_q^p(m+1)), \text { otherwise.} } \end{cases}$$
\end{cor}

\begin{remarkcounter} \label{l=2}
Consider the case $\ell = 2$. Here we have $T_q^p(1) = V_q^p$. So Lemma \ref{small tilt times V} says $T_q^p(1) \otimes V_q^p \simeq T_q^p(2)$ and $T_q^p(2) \otimes V_q^p \simeq T_q^p(3) \oplus T_q^p(1)^{\oplus 2}$. Proposition \ref{q-not-2} says that if $m$ is odd then  $T_q^p(m) \otimes V_q^p \simeq T_q^p(m+1)$, whereas Corollary \ref{q-not-1} states that for all even $m>0$ we have $((V_q^p)^{\otimes n} : T_q^p(m)) =  ((V_q^p)^{\otimes n-1} : T_q^p(m-1)$.
\end{remarkcounter}

To handle the remaining case we need to decompose $T_q^p(m) \otimes V_q^p$ also for  $m \equiv -2 \; (\mo \;\ell)$.  We first observe that $T_q^p(\ell - 2) \otimes V_q^p  \simeq T_q^p(\ell -1) \oplus T_q^p(\ell - 3)$ and $T_q^p(2\ell - 2) \otimes V_q^p  \simeq T_q^p(2\ell -1) \oplus T_q^p(2\ell - 3)$ where the first summand may also be written $V^{[q]} \otimes T_q^p(\ell - 1)$.  In general we get by combining this with Proposition \ref{q-Donkin}  

\begin{prop} \label {-2}
Let $m = s \ell - 2$ with $s > 1$. Then
\begin{enumerate}
\item if $\ell >2$ then 
$T_q^p(m) \otimes V_q^p \simeq (T(s-2) \otimes V)^{[q]} \otimes T_q^p(\ell -1) \oplus T_q^p(m-1)$, 
\item if $\ell = 2$ then $T_q^p(m) \otimes V_q^p \simeq (T(s-2) \otimes V)^{[q]} \otimes T_q^p(\ell -1) \oplus T_q^p(m-1)^{\oplus 2}.$
\end{enumerate}
\end{prop}

When we combine this with Proposition \ref{not-2} we obtain

\begin{cor} \label{q-2}
Let $m = (s+2) \ell - 2$ with $s \geq 0$ and suppose $s$ is not divisible by $p$. Then for $\ell > 2$ we get 
$$ T_q^p(m) \otimes V_q^p \simeq \begin{cases} {T_q^p(m+1) \oplus T_q^p(m-1) \text { if } s \equiv -1 \; (\mo \; p)}\\   {T_q^p(m+1) \oplus T_q^p(m+1 - 2\ell) \oplus T_q^p(m-1)  \text { otherwise. } }\end{cases} $$
When $\ell = 2$ we get 
$$ T_q^p(m) \otimes V_q^p \simeq \begin{cases} {T_q^p(m+1) \oplus T_q^p(m-1)^{\oplus 2} \text { if } s \equiv -1 \; (\mo \; p)}\\   {T_q^p(m+1) \oplus T_q^p(m+1 - 2\ell) \oplus T_q^p(m-1)^{\oplus 2}  \text { otherwise. } }\end{cases} $$
\end{cor}

\begin{cor}
Let $m = (s+2) \ell - 1$ and suppose $s$ is not divisible by $p$. Then for all $\ell$
$$ ((V_q^p)^{\otimes n}:T_q^p(m)) = ((V_q^p)^{\otimes n-1}:T_q^p(m-1)) + 2 ((V_q^p)^{\otimes n-1}:T_q^p(m+1)) + R$$
where $R = \begin{cases} {0 \text { if } s \equiv -1 \; (\mo \; p), } \\  {((V_q^p)^{\otimes n-1}:T_q^p(m-1 + 2\ell)) \text { otherwise.}} \end{cases} $

\end{cor}

The case where $s$ is divisible by $p$ is handled by appealing to Proposition \ref{=2p-2} . In this case we get

\begin{cor}
Let $m = s \ell - 2$ with $s > 1$ and suppose $s = b p^r$ for some $b$ prime to $p$ and $r> 0$.  If $\ell > 2$ 
$$ T_q^p(m) \otimes V_q^p \simeq T_q^p(m+1) \oplus (\bigoplus_{j=0}^{r-1} T_q^p(m+1 - 2p^j \ell)) \oplus T_q^p(m+1 -2p^r \ell)^{\oplus a} \oplus T_q^p(m-1) $$
where $a = 0$ if $ b \equiv -1 \; (\mo \; p)$ and $a = 1$ in all other cases.

When $\ell = 2$ we get the same formula except for the last term which in that case occurs twice.
\end{cor}

This lead to the following recurrence relation for the tilting multiplicities $((V_q^p)^{\otimes n} : T_q^p(m))$ in the case when $m = s \ell -1$ for some $s$ divisible by $p$.

\begin{cor} 
Let $n, m \in \Z_{>0}$ and suppose $m = bp^r \ell -1$ for some $r \geq 1$ and $b$ prime to $p$. Then for all $\ell$
$$ ((V_q^p)^{\otimes n}:T_q^p(m)) = ((V_q^p)^{\otimes n-1}:T_q^p(m-1)) + 2  ((V_q^p)^{\otimes n-1}:T_q^p(m+1)) +$$ 
$$ {  \sum_{s=0}^{r-1}  ((V_q^p)^{\otimes n-1}:T_q^p(m-1 + 2p^s\ell )) + \begin{cases} {0 \text { if } b \equiv -1 \; (\mo \; p) } \\ {((V_q^p)^{\otimes n-1}:T_q^p(m-1 + 2p^r \ell)) \text { otherwise.} }\end{cases}} $$

\end{cor}

\begin{examplecounter}
We have illustrated the above algorithms by using them to calculate the tilting multiplicities when $(\ell, p) = (2,3)$, see Table  6, respectively $(\ell, p) = (3,2)$, see Table 7.  Note that for any pair $(\ell, p)$ the first $p\ell + \ell -1$ rows in such figures will coincide with the same rows in the case where $p = 0$ (compare e.g. the first  8 rows in Table  3 with the ones in Table  7. We have chosen minimum values of $(\ell, p)$ in order to see that the numbers do indeed differ in general and also in order to avoid too large numbers in our figures.

Note that in Table  6 we have in analogy with Table  4 only listed the odd number values. The even counterparts are then found via Remark \ref{l=2}.
\eject 

\centerline
{ \it Table  6.  Tilting multiplicities in $(V_q^p)^{\otimes n}$ for $\ell = 2$ and $p=3$}
\vskip .5cm
\noindent
\begin{tabular}{ r| c  c c c c c c c c c c c c c c c c}
  
   & 1 & 3 & 5 & 7 & 9 & 11 & 13 & 15 & 17& 19 &   \\  \hline 
  
  1 & 1 &  \\ 
  3 & 2 &  1 \\ 
  5 & 5& 4 & 1   \\
  7 & 14 &13 & 6 & 1\\
9&41 &40 & 27& 8 & 1 & \\
11 &122 &121& 110&44 &10  &1  \\
13&365 &364 &429 &208&64&12&1 \\
15& 1094&1093&1638&909&336 &90&14 & 1& \\
17&3281&3280&6188&3792&1581&544&119&16&1\\
19 &9842&9841&23256&15353&6954&2907&798&151&18&1\\

\end{tabular}

\vskip 1cm
\centerline
{ \it Table   7.  Tilting multiplicities in $(V_q^p)^{\otimes n}$ for $\ell = 3$ and $p=2$}

\vskip .5 cm
\noindent
\begin{tabular}{ r| c  c |c c c| c c c | c c c| c c c |c c c}
  
   & 0 & 1 & 2 & 3 & 4 & 5 & 6 & 7 & 8 & 9 & 10 & 11 & 12 & 13 & 14 & 15 & 16  \\  \hline 
  0 & 1  &&&&&&&&&&&&&&&  \\ 
  1 &  & 1&&&&&&&&&&&&&&& \\ 
  2 & 1 & & 1&&&&&&&&&&&&&& \\ 
  3 & & 1 &   & 1 &&&&&&&&&&&&&\\ 
  4 & 1 & & 3 & & 1&&&&&&&&&& \\
5 & &1 & & 4 & & 1 &&&&&&&&& \\
6 &1 && 9& &4 & &1  &&&&&&&&&\\
7 & &1 & &13&&6&&1&&&&&& \\
8 & 1&&27 &&13 &&7 & &1 &&&&&&\\
9 &&1&&40&&27&&7&&1&&&&&&\\
10 &1&&81&&40&&34&&9&&1&&&&&&\\
11&&1&&121&&110&&34&&10&&1&&&&\\
12 &1&&243&&121&&144&&54&&10&&1&&&\\
13 &&1&&364&&429&&144&&64&&12&&1&&\\
14 &1&&729&&364&&573&&272&&64&&13&&1&\\
15&&1&&1093&&1638&&573&&336&&90&&13&&1&\\
16 &1&&2187&&1093&&2211&&1245&&336&&103&&15&&1

\end{tabular}
\end{examplecounter}

\subsection{$q$-Steinberg class multiplicities in $(V_q^p)^{\otimes n}$}

As in Section 4 we shall give an alternative way of finding the tilting multiplicities $((V_q^p)^{\otimes n}:T_q^p(m))$ for those $m$ which are congruent to $-1$ modulo $\ell$. Since the Weyl multiplicities of $(V_q^p)^{\otimes n}$ are known we need to write (in the Grothendieck group) the relevant Weyl modules as $\Z$-linear combinations of indecomposable tilting modules.

The dot-action of $\ell$ on $\Z$ is given by $\ell \cdot r = \ell(r+1) - 1$. Then $\ell \cdot \Z_{\geq 0}$ is the set of non-negative integers congruent to $-1$ modulo $\ell$. In analogy with (\refeq{Steinberg Weyl modules}) and  (\refeq{Steinberg tilting modules}) we have (with $St_q^p = \Delta_q^p(\ell -1)$)

\begin{equation} \label{q-Steinberg Weyl modules}
\Delta_q^p(\ell\cdot m) \simeq  \Delta(m)^{[q]} \otimes St_q^p,
\end{equation}
and
\begin{equation} \label{q-Steinberg tilting modules}
T_q^p(\ell\cdot m) \simeq  T(m)^{[q]} \otimes St_q^p.
\end{equation}

Using (\refeq{q-Steinberg Weyl modules}) and  (\refeq{q-Steinberg tilting modules}) we get arguing exactly as in Section 5.5 (and using the notation $a_{n,m}$ and $d_{r,s}$ from there)
\begin{prop} \label{q-Steinberg-mult}
Let $n,s \in \Z_{\geq 0}$. Then $((V_q^p)^{\otimes n}:T_q^p(\ell\cdot s)) = \sum _r  a_{n,\ell \cdot r} d_{r,s}$ where the sum runs over those $r$ which satisfy
 $r \geq s$, $\ell \cdot r \leq n$ and $s \in (r + 2p\Z) \cup (r' + 2p\Z)$.
\end{prop}

\subsection{Fusion}
Again in this case we have a fusion quotient of the category of tilting modules for $U_q$. It is defined by dividing out by the ideal generated by $T_q^p(m)$ with $m \geq \ell -1$ (note that this is indeed a tensor ideal by Proposition \ref{q-not-2} and Corollary \ref{q-2}). We leave to the reader the task of going through the steps in Sections 3.3 and 4.8.
The conclusion is that we have a fusion tensor category with tensor product denoted $\underline \otimes$ satisfying

\begin{cor}. Let $q_0$ be a root of unity in a field of characteristic $0$ and suppose its order is $\ell$. Then we have for all $n \in \Z_{\geq 0}$ and all $m \in [0, \ell -2]$
$$ ((V_q^p)^{\underline \otimes n}: T_q^p(m)) =  (V_{q_0}^{\underline \otimes n}: T_{q_0}(m)). $$
\end{cor}

\section{Modules for the Temperley-Lieb algebras}

Let $n \in Z_{\geq 0}$ and denote by $TL_n(q + q^{-1})$ the Temperley-Lieb algebra on $n$ strands with parameter $ q + q^{-1}$. Here $q$ can be an arbitrary element in $k\setminus \{0\}$ but as we shall see the module category for $TL_n(q + q^{-1})$ is most interesting when $q$ is a root of unity.

\subsection{Cell modules for Temperley-Lieb algebras}

It is wellknown that $TL_n(q + q^{-1}) \simeq \End_{U_q}(V_q^{\otimes n})$, see \cite{AT}. Here $U_q = U_q(sl_2)$ is the quantum group for the Lie algebra $sl_2$ and $V_q$ is the $2$-dimensional simple module for $U_q$. As we have seen $V_q^{\otimes n}$ is always (i.e. for all $k$ and all $q$) a tilting module for $U_q$. Therefore $\End_{U_q}(V_q^{\otimes n})$ is a cellular algebra by \cite {AST1}. The cell modules for this algebra are
$$C_n^q(m) = \Hom_{U_q}(\Delta_q(m), V_q^{\otimes n}), \; m \in [0, n] \cap (n + 2\Z).$$
Set now $\Lambda_n =  [0, n] \cap (n + 2\Z)$. By standard tilting theory (see e.g \cite{AST1}) this gives via  \refeq{Weyl mult} the (wellknown) formula

\begin{prop} \label{celldim}
Let $m \in  \Lambda_n$ and set $r = (n-m)/2$. Then
$$ \dim C_n^q(m) = (V_q^{\otimes n} : \Delta_q(m)) = \binom{n}{r} - \binom{n}{r-1}.$$
\end{prop}
  
\begin{remarkcounter}
Observe that the dimension formula in this proposition is valid for all $k$ and all $q$. Note that both $TL_n(q + q^{-1})$ and $U_q$ have integral versions, i.e. there are algebras $TL_n^{A}(v + v^{-1})$, respectively $U_A$, over $A = \Z[v,v^{-1}]$ with $TL_n(q + q^{-1}) \simeq TL_n^{A}(v+v^{-1})\otimes_A k$ and $U_q \simeq U_A \otimes_A k$. Here $k$ is considered an $A$-algebra via $v \mapsto q$. The Weyl modules have integral versions $\Delta_A(m)$ (in particular $V_A = \Delta_A(1)$) which are free over $A$. From this we see that also the cell modules $C_{n}^A(m) = \Hom_{U_A}(\Delta_A(m), V_A^{\otimes n})$ for $\End_{U_A}(V_A^{\otimes n})$ are free over $A$ and satisfy $C_{n}^A(m) \otimes_A k \simeq C_n^q(m)$. This explains that the dimensions of the cell modules are independent of $k$ and $q$.

The cell modules for $TL_n(q+q^{-1})$  all have unique simple heads. By \cite{AST1} the module category for $TL_n(q + q^{-1}) \simeq \End_{U_q}(V_q^{\otimes n})$ is semisimple iff $V_q^{\otimes n}$ is semisimple as a module for $U_q$. This is the case for all $n$ when $q$ is not a root of unity in $k$. It is also true if $q = \pm 1$ and $k$ has characteristic $0$ (in both these cases all finite dimensional modules for $U_q$ are semisimple). In these cases the cell modules are therefore simple and give up to isomorphisms all simple modules for $TL_n(q + q^{-1})$. Their dimensions are in this case given by Proposition \ref{celldim}.

\end{remarkcounter}

\subsection{Dimensions of simple modules}

We denote for $m \in \Lambda_n$ by $D_n^q(m)$ the head of $C_n^q(m)$. Then these $D_n^q(m)$'s constitute up to isomorphisms the simple modules for $TL_n(q + q^{-1})$. We shall provide algorithms which determine the dimensions of these simple modules. In the case where $q$ is a non-root of unity we have already obtained a closed formula  for these dimensions. So in the following we assume that $q$ is a root of unity. We have the important identity, see Theorem 4.12 in \cite{AST1}. 
\begin{prop} \label{dim-formula}
If $m \in \Lambda_n$ then $\dim D_n^q(m) = (V_q^{\otimes n} : T_q(m)).$
\end{prop}

We now separate into the three different cases considered in Sections 3-5.

\subsubsection{Non-trivial roots of unity in characteristic $0$}

Consider a field $k$ of characteristic $0$ and $q \in k$ a root of unity. Denote by $\ell$ the order of $q^2$. We shall assume $\ell > 1$ as otherwise we are in the semisimple case discussed above. We can then describe the algorithm obtained i Section 3 as follows.

Let $P = (p_{n,m})$ denote the non-negative half of the Pascal triangle. Its entries are given by  $p_{0,m} = \delta_{0,m}$, $p_{n,m}= 0$ for $m < 0$,  and for $n>0$ we have $p_{n,m} = p_{n-1,m-1}  + p_{n-1,m+1}$. It is then easy to check that in fact
$p_{n,m} = \binom{n}{r} - \binom{n}{r-1}$ with $r = (n-m)/2$. In other words (see Proposition \ref{celldim}) we have $p_{n,m} = \dim C_n^q(m)$, the dimension of the $m$'th cell module for $TL_n(q + q^{-1})$.

We shall now construct a quantum $\ell$-version $P_q(\ell) = (p_{n,m}^q(\ell))$ of $P$:

First we set $p_{0,m}^q(\ell) = p_{0,m}$ for all $m$ and $p_{n,m}^q(\ell) = 0$ if $m < 0$. If $n>0$ we set 

$$ p_{n,m}^q(\ell) = \begin{cases} {p_{n,m} \text { if } m>0  \text  { with } m\equiv -1 \; (\mo \; \ell)} \\ {p_{n-1,m-1}^q(\ell)  \text { if } m>0 \text  { with } m\equiv -2 \; (\mo \; \ell)} \\ {p_{n-1,m-1}^q(\ell)  + p_{n-1,m+1}^q(\ell) \text { otherwise. }} \end{cases} $$

Then

\begin{prop}
The dimensions of the simple modules are given by $\dim D_n^q(m) = p_{n,m}^q(\ell)$ for all $m \in \Lambda_n$.
\end{prop}
\begin{proof} This is Proposition \ref{dim-formula} combined with the algorithm in Corollary \ref{q-tilt recurrence}.

\end{proof}

\begin{examplecounter}
Suppose $\ell = 5$. Then this proposition gives that the simple modules for $TL_{16}(q + q^{-1})$ have dimensions: $610, 987, 3640, 2445, 820, 440, 103, 15, 1$, cf. the bottom row in Table  2. If instead we consider $\ell = 3$ then the corresponding dimensions are $1, 3432, 1429, 2211, 1260, 337, 103, 15, 1$ as is seen from Table  3.
\end{examplecounter}

\subsubsection{$q=1$ in characteristic $p>0$}
Consider a field $k$ of characteristic $p>0$ and let $q = 1 \in k$. We  write $C_n(m)$ for the cell modules and $D_n(m)$ for the simple modules for $TL_n(2)$. This time we need a $p$-version $P(p) = (p_{n,m}(p))$ of the Pascal triangle. It is given by: First we set $p_{0,m}(p) = p_{0,m}$ for all $m$ and $p_{n,m}(p) = 0$ if $m < 0$. If $n>0$ we set 

$$ p_{n,m}(p) = \begin{cases} {\sum_{r \geq 0}p_{n,p\cdot r}d_{r,s} \text { if } m = p\cdot s >0  \text  { for some } s \geq 0, } \\ {p_{n-1,m-1}(p)  \text { if } m>0 \text  { with } m\equiv -2 \; (\mo \; p),} \\ {p_{n-1,m-1}(p)  + p_{n-1,m+1}(p) \text { otherwise. }} \end{cases} $$
(here the numbers  $d_{r,s}$  are those introduced in Section 4.7).

This time we have 
\begin{prop}
The dimensions of the simple modules for $TL_n(2)$ are given by $\dim D_n(m) = p_{n,m}(p)$ for all $m \in \Lambda_n$.
\end{prop}
\begin{proof} This is Proposition \ref{dim-formula} combined with Corollary \ref{not -1} and Proposition \ref{Steinberg mult}.

\end{proof}

\begin{remarkcounter}
Alternatively, we could use the algorithm given in Section 4.6 to calculate the simple dimensions. Above we have chosen the one from Section 4.7 in order to stress the similarity to the quantum root of unity case in characteristic zero.
\end{remarkcounter}
 
\begin{examplecounter}
Suppose $p=3$. Then we get from this proposition and Table  5 the following dimensions for the simple modules for $TL_{16}(2)$: $1, 3417, 1428, 2108, 1260, 337, 103, 15, 1$. 

If instead $p = 2$ we find the relevant tilting multiplicities in Corollary \ref{recurrence p=2}. Via Table  4 we see that in the case of $TL_{16}(0)$ the list of dimensions of simple modules are: $128, 1912, 1288, 910, 336, 90, 14, 1$.
\end{examplecounter}

\subsubsection{Non-trivial roots of unity in characteristic $p>0$}

Consider a field $k$ of characteristic $p>0$ and let $q \in k$ be a root of unity. Set as usual $\ell$ equal to the order of $q^2$, and assume $\ell >1$. This time we write $C_n^{q,p}(m)$ for the cell modules and $D_n^{q,p}(m)$ for the simple modules for $TL_n(q+q^{-1})$. In analogy with the above we let now  $P(\ell, p) = (p_{n,m}(\ell, p))$ be the deformed Pascal triangle given by:  $p_{0,m}(\ell, p) = p_{0,m}$ for all $m$,  and $p_{n,m}(\ell, p) = 0$ if $m < 0$. Moreover, if $n>0$ then 
$$ p_{n,m}(\ell, p) = \begin{cases} {\sum_{r \geq 0}p_{n,\ell \cdot r}d_{r,s} \text { if } m = \ell \cdot s >0  \text  { for some } s \geq 0, } \\ {p_{n-1,m-1}(\ell, p)  \text { if } m>0 \text  { with } m\equiv -2 \; (\mo \; \ell),} \\ {p_{n-1,m-1}(\ell p)  + p_{n-1,m+1}(\ell, p) \text { otherwise. }} \end{cases} $$
(here the numbers  $d_{r,s}$  are those introduced in Section 4.7).

Again we get
\begin{prop}
The dimensions of the simple modules for $TL_n(q+q^{-1})$ are given by $\dim D_n^{q, p}(m) = p_{n,m}(\ell, p)$ for all $m \in \Lambda_n$.
\end{prop}
\begin{proof} This is Proposition \ref{dim-formula} combined with Corollary \ref{q-not-1} and Proposition \ref{q-Steinberg-mult}.

\end{proof}
 
\begin{examplecounter}
Suppose  $(\ell, p) =(2,3)$. Then we can via this proposition read off from Table  6 the dimensions for all simple modules for $TL_n(q+q^{-1})$ for $n \leq 20$. Similarly if $(\ell, p) = (3,2)$ then Table  7 contains analogous information.
\end{examplecounter}

\section{Simple modules for the Jones quotient algebras}

In this section we shall assume either that $k$ has characteristic $p \geq 0$ and that $q \in k$ is a root of unity with $\ord (q^2) = \ell >1$ or that $k$ has characteristic $p>0$ and $q = 1$. In these cases the Temperley-Lieb algebras $TL_n(q + q^{-1})$ have interesting semisimple quotients $Q_n(q + q^{-1})$ known as the Jones algebras, \cite{ILZ}. They are defined as the quotients by the Jones-Wenzl idempotent in $TL_{\ell - 1}(q+q^{-1})$, see loc.cit. For us the most convenient definition of these algebras are 
$$ Q_n(q + q^{-1}) = \End_{U_q}(V_q ^{\underline \otimes n}).$$
Here $\underline \otimes$ is the "reduced" tensor product on the fusion category, see Sections 3.3, 4.8 and 5.12.

Recall that $(D_n(m))_{m \in \Lambda_n} $ is the family of simple modules for $TL_n(q+q^{-1})$. We set 
$$d_n(m) = \begin{cases} {\dim D_n(m)  \text { if } m \in \Lambda_n \cap [0, \ell -2],}\\{0 \text { otherwise.}} \end{cases}$$
Here we replace $\ell$ by $p$ when $q = 1$.

 Then

\begin{prop}
The simple modules $L_n(m)$ with $m  \in \Lambda_n \cap [0, \ell -2]$ are modules for $Q_n(q + q^{-1})$ and are up to isomorphisms the list of simple modules for $Q_n(q + q^{-1})$. Their dimensions satisfy (and are determined by)
$$ d_n(m) = d_{n-1}(m-1) + d_{n-1}(m+1).$$
\end{prop}

\begin{proof} The first statement follows from the definition of $Q_n(q + q^{-1})$ as the quotient $ \End_{U_q}(V_q ^{\underline\otimes n})$ of $ \End_{U_q}(V_q ^{\otimes n}) = TL_n(q + q^{-1})$. The second statement is a special case of Corollaries 3.7,  4.14 and 5.12.
\end{proof}

\begin{remarkcounter}
In characteristic $0$ this proposition is wellknown, see \cite{ILZ}.
\end{remarkcounter}

\begin{examplecounter}

If $\ord(q^2) = 7$ or if $q=1$ and the characteristic of $k$ is $7$ then the dimensions of the simple modules for the Jones algebras $Q_n(q + q^{-1}), \; n \leq 16$  are found in the first $6$ columns in Table  1. If we replace $7$ by $3$ we see that the unique simple modules of the Jones algebra (for all $n$)  is $1$-dimensional (cf. Table 5). In fact, the Jones algebras are all trivial in this case.
\end{examplecounter}

\section{Some other endomorphism algebras}

We return for a little while to the case where $k$ is an arbitrary field of characteristic $p \geq 0$  and $q$ an arbitrary non-zero element in $k$. Let $T_q$ denote a tilting module for $U_q$ and set $E_q = \End_{U_q}(T)$. Then by the general theorem, \cite{AST1} Theorem 3.9 we know that $E_q$ is a cellular algebra. The cell modules for $E_q$  are $C_q(m) = \Hom_{U_q} (\Delta_q(m), T)$ and the simple modules for $E_q$ are the heads $L_q(m)$ of those $C_q(m)$ for which $(T_q:\Delta_q(m)) \neq 0$, see \cite{AST1} Theorem 4.11. The dimension of $C_q(m)$ is the Weyl factor multiplicity $(T_q: \Delta_q(m))$ and the dimension of $L_q(m)$ is the tilting multiplicity $(T_q: T_q(m))$, see \cite{AST1} Section 4. 

\subsection{Methods}
We have the following methods for computing the dimensions above.
\vskip .25 cm
{\bf Method 1}. 
Recall that the classes $([\Delta_q(m)])_{m \in \Z_{\geq 0}}$ form a basis of the Grothendieck group $\mathcal K$, see Section 3.1. Hence we can write $[T_q] = \sum_m a_m [\Delta_q(m)]$ for unique integers $a_m$. Then these numbers coincide with the Weyl multiplicities $(T_q:\Delta_q(m))$, i.e. we have $\dim C_q(m) = a_m$.  Moreover, the classes $([T_q(m)])_{m \in \Z_{\geq 0}}$ likewise constitute a basis for $\mathcal K$ and the results in Sections 3-5 tell us how to calculate the numbers $d_{i,j}$ determined by $[\Delta_q(j)] = \sum_i d_{j,i} [T_q(i)]$. Then
\begin{equation}
\dim L_q(m) = \sum_r a_r d_{r,m}.
\end{equation}

{\bf Method 2}.
We have a third basis of $\mathcal K$, namely the classes $([V_q^{\otimes r}])_{r \in Z_{\geq 0}}$. Hence there are also unique integers $b_r$ such that $[T_q] = \sum_r b_r[V_q^{\otimes r}]$. In Sections 3-5 we presented recipes for calculating the matrix $c_{r,m} = (V_q^{\otimes r} : T_q(m))$. With this notation we have
\begin{equation} \label{method 2}
\dim L_q(m) = \sum_r b_r c_{r,m}.
\end{equation}

\vskip .5 cm

Of course, the first steps in these two methods are equivalent. In fact, the numbers $a_m$ in Method 1 are related to the numbers $b_n$ in Method 2 by the equation 
$$ a_m =  \sum_n b_n (\binom{n}{(n-m)/2} - \binom{n}{(n-m)/2 -1})$$ 
with inverse
$$ b_n =  \sum_m a_m (-1)^{(m-n)/2}\binom{(m+n)/2}{n}$$
where we interprete $\binom{r}{j}$ as $0$ unless $r \in \Z$ and $j \in \Z_{\geq 0}$. Here the first equation comes from (\refeq{Weyl mult}) and the second from the (easily checked) inverse relation $ [\Delta_q(m)] = \sum_{j=0}^m (-1)^{(m-j)/2}\binom{(m+j)/2}{j} [V_q^{\otimes j}]$ in $\mathcal K$.
\vskip .5 cm

We can also consider the family of tensor powers $(T_q^{\otimes n})_{n \geq 0}$. Then we set $E_q^n = \End_{U_q} (T_q^{\otimes n})$ and denote by $C_q^n(m)$, respectively $L_q^n(m)$, the cell module, respectively the simple module for the cellular algebra  $E_q^n$. In this case, we have in addition to the above two methods an inductive procedure:
\vskip .5 cm
{\bf Method 3}.
Let $a_m$ denote the integers determined in Method 1 and recall the Clebsch-Gordan tensor formula
\begin{equation} \label{Clebsch-Gordan}
(\Delta_q(r) \otimes \Delta_q(m) : \Delta_q(s)) = \begin{cases} {1 \text { if } |m-r| \leq s \leq r+ m ,\text { and } s \equiv r+m \;  (\mo \; 2)}\\ {0 \text { otherwise.}}\end{cases}
\end{equation}
Then we get (for $n>0$) $(T_q^{\otimes n}: \Delta_q(m)) = \sum_s (T_q^{\otimes n-1} : \Delta_q(s)) (T_q \otimes \Delta_q(s) : \Delta_q(m)) = \sum_{s,r} (T_q^{\otimes n-1} : \Delta_q(s)) a_r (\Delta_q(r) \otimes \Delta_q(s) : \Delta_q(m)) = \sum_s \sum_{j = 0}^{\min\{s,m\}} (T_q^{\otimes n-1}: \Delta_q(s)) a_{|m-s|+ 2j}$. Equivalently, we have the inductive recipe for the dimensions of the cell modules
\begin{equation} \label{ind cell}
\dim C_q^n(m) = \sum_s \sum_{j=0} ^{\min\{s,m\}} \dim C_q^{n-1}(s) a_{m-s+ 2j}.
\end{equation}
Proceding as in Methods 1 we can finally use this information to calculate the dimension of the simple modules
\begin{equation}  \label{dim simples}
\dim L_q^n(m) = \sum_r \dim C_q^n(r) d_{r,m}.
\end{equation}

\subsection{The case of tilting Weyl modules}

Fix an $m \in \Z_{> 0}$ and assume throughout this section that $T_q(m) = \Delta_q(m)$. Note that this assumption is true for all $m$ if $q$ is not a root of unity or if $q= \pm 1$ and $p=0$. If $q \neq \pm 1$ is a root of unity we write as before $\ell = \ord(q^2)$. Then our assumption holds e.g. when $m < \ell$. It always hold when $m=1$ - the case treated in Sections 3-5. For $m < \ell$ this situation was also studied in \cite{ALZ}.

We shall consider the tensor powers of $\Delta_q(m)$. For each $n \in \Z_{\geq 0}$ we want to examine the $k$-algebra
$$ E_q^n (m) = \End_{U_q}(\Delta_q(m)^{\otimes n})$$
and its simple modules. Of course these modules depend on both $q$ and $p$.

If $q$ is not a root of unity then (as we have observed before) all finite dimensional modules for $U_q$ are semisimple. The same is true if $q = \pm 1$ and $p= 0$. In these cases the simple modules for $E_q^n(m)$ therefore coincide with the cell modules $C_q^n(m, s) = \Hom_{U_q}(\Delta_q(s), \Delta_q(m)^{\otimes n})$. Their dimensions are inductively given by the following general formula, see formula (4.8) in \cite{ALZ}. Note that $\dim C_q^0(m, s) = \delta_{s,0}$.

\begin{prop} \label{cell recurs}
If $n >0$ then we have
$\dim C_q^n(m, s) = \sum_{j=0}^{\min\{s,m\}} \dim C_q^{n-1}(m, |s-m|+2j)$ for all $s$.

\end{prop}

\begin{proof}
This is formula (\refeq{ind cell}) applied to the case $T_q = \Delta_q(m)$.
\end{proof}

\begin{prop} \label{ss}
Suppose $m < \ell$. The algebra $E_q^n(m)$ is semisimple if and only if $nm < \ell$. Hence the cell modules $C_q^n(m,s)$ are simple for all $s$ if and only if $mn < \ell$.
\end{prop}
\begin{proof}
We know that $E_q^n(m)$ is semisimple if and only if the module $\Delta_q(m)^{\otimes n}$ is semisimple. If $mn < \ell$ then all weights of $\Delta_q(m)^{\otimes n}$ are less than  $\ell$ and in this case $\Delta_q(m)^{\otimes n}$ is clearly semisimple. On the other hand, $T_q(mn)$ always occurs as a summand of $\Delta_q(m)^{\otimes n}$ and if $mn \geq \ell$ it follows from our descriptions in Sections $3$ and $5$ that $T_q(mn)$ is non-semisimple  unless possibly when $mn \equiv -1 \; (\mo \; \ell)$. However, also $T_q(mn-2)$ occurs as a summand of $\Delta_q(m)^{\otimes n}$ (in fact with multiplicity $n-1$ and since $nm \geq \ell$ we have $n>1$) and if $mn \equiv -1 \; (\mo \; \ell)$ then $T_q(mn -2)$ is non-semisimple (here we use $\ell > 2$).
\end{proof}

Once we have calculated the dimensions of the cell modules equation (\refeq{dim simples}) tells us that the dimensions of the simple modules $L_q^n(m, s)$ for $E_q^n(m)$ are given  by the formula
\begin{equation}
\dim L_q^n(m, s) =  \sum_r \dim C_q^n(m, r) d_{r, s}
\end{equation}

Observe that the matrix $(d_{r,s})$ is determined in Sections 3-5. We illustrate by giving details in the simplest case, namely the characteristic zero root of unity case (cf. Section 3).

\begin{examplecounter}
Let $p = 0$ and suppose either that $m < \ell$ or that $m \equiv -1 \; (\mo \; \ell)$ (so that our assumption $\Delta_q(m) = T_q(m)$ is satisfied).  Then we have
\begin{enumerate}
\item If $s \equiv -1 \; (\mo \; \ell)$ then  $L_q^n(m, s) =   C_q^n(m, s)$,
\item If $s = s_0 + s_1 \ell$ with $0 \leq s_0 \leq \ell -2$ then 
$$\dim L_q^n(m, s) =  \sum_{j\geq 0}  \dim C_q^n(m, s + 2j\ell) -  \sum_{j> 0} \dim C_q^n(m, s' + 2j\ell),$$
where $s' = s_1\ell - s_0 -2$.
\end{enumerate}
\end{examplecounter}

In the following section we shall work out further this example in the case where $m=2$.

\section{BMW-algebras}

In this section we shall examine closely the simple modules for the endomorphism rings $E_q^n(2)$, i.e. in the set-up of the previous section we take $m=2$. The three dimensional Weyl module $\Delta_q(2)$ is tilting unless we are in a field of characteristic $2 $ and $q = 1$ or $q$ is a $4$'th root of unity (in any field). So we shall exclude these cases together with all cases where $q$ is either a non-root of unity in an arbitrary field or $q= \pm 1$ in a field of characteristic $0$ (in the latter cases the algebras $E_q^n(2)$ are semisimple for all $n$ and the simple modules coincide with the cell modules). In other words, we look in this section at the cases where $q =\pm 1$ in a field of characteristic $p > 2$ together with the cases where $q$ is a root of unity with $\ell = \ord(q^2) > 2$ in an arbitrary field. 

It is wellknown that the algebras $E_q^n(2)$ are closely related to (a particular specialization, see below) of the family of BMW-algebras. We shall conclude this section by giving the dimensions for a class of simple modules for BMW-algebras  which  we can extract from our results. It turns out that for this we need  in addition to the above assumptions also that $q$ is not a $6$'th root of $1$.

\subsection{Simple modules for $E_q^n(2)$}

We set $W= \Delta_q(2)$ and we shall write $C_W^n(s)$ and  $L_W^n(s)$ for the cell module and the simple module for $E_q^n(2)$ labeled by $s$. Note that the weights of $W$ are $-2, 0, 2$. In particular the weights are all even. This means that for any $n$ the relevant $s$ to consider are also all even. 

In this case Proposition \ref{cell recurs} states (for $n, s >0 $ with $s$ even)
\begin{equation}\label{recurs cell}
\dim C_W^n(0) = \dim C_w^{n-1} (2) \text { and }  \dim C_W^n(s) = \dim C_W^{n-1} (s-2) +  \dim C_W^{n-1} (s) +  \dim C_W^{n-1} (s+2).
\end{equation}
Alternatively, we could use Method 2 in Section 8.1: We note that in the Grothendieck group we have $[W] = [V^{\otimes 2}] - [k]$. Hence 
\begin{equation} \label{W in terms of V}
[W^{\otimes n}] = \sum_{i=0}^n (-1)^i \binom{n}{i} [V^{\otimes 2(n-i)}].
\end{equation}
This gives via equation (\refeq{Weyl mult})
the following closed formula for $(W^{\otimes n}: \Delta_q(s)) = \dim C_W^n(s)$.
\begin{equation} \label{dim cell}
\dim C_W^n(s) = \sum_{i=0}^n (-1)^i \binom{n}{i} (\binom{2(n-i)}{n-i-s/2} -\binom{2(n-i)}{n-i-1-s/2}).
\end{equation}

Using either of these equations one may easily calculate the dimensions of the cell modules for $E_q^n(2)$. In Table 8 we have (in the $n$'th row) recorded those numbers for $n \leq 10$. 
\vskip .5 cm 

{Table 8. Dimensions of cell modules for $E_q^n(2)$ for $n= 0, 1, \cdots , 10$.} 
\vskip .5 cm
\noindent
\begin{tabular}{ r| c  c c c c c c c c c c c c c c c c}
 
   & 0 & 2 & 4 & 6 & 8 & 10 & 12 & 14 & 16& 18 & 20  \\  \hline 
  0 & 1 &  \\ 
  1 &  &  1 \\ 
  2 & 1& 1 & 1   \\
  3 & 1 &3 & 2 & 1\\
4&3 &6 & 6& 3 & 1 & \\
5 &6 &15& 15&10 &4  &1  \\
6&15 &36 &40 &29&15&5&1 \\
7& 36&91&105 &84&49&21&6 & 1& \\
8&91&232&280&238&154&76&28&7&1\\
9 &232&603&750&672&468&258&111&36&8&1\\
10 &603&1585&2025&1890&1398&837&405&155&45&9&1\\

\end{tabular}
\vskip .5 cm

We can now find the dimensions of the simple modules by combining equation (\refeq{dim cell}) (or (\refeq{recurs cell})) and (\refeq{method 2}).  Alternatively, we can apply the results from Sections 3, respectively 4 and 5) on the decompositions of the tensor powers of $V_q$ (respectively $V$ and $V_q^p$) in combination with (\refeq{method 2}) to obtain
\begin{equation} \label{simples via TL}
\dim L_W^n(s) =  \sum_{i=0}^n (-1)^i \binom{n}{i} c_{2(n-i),s}.
\end{equation}
Here the matrix $(c_{r,s})$ is the matrix containing the tilting multiplicities in the tensor powers of $V$ (respectively, $V$, $V_q^p$), i.e. the dimensions of the simple modules for the Temperley-Lieb algebras. The algorithms in Sections 3.2, 4.5-7, and 5.4-5 show how to calculate this matrix.

\begin{examplecounter}
Again we illustrate the above formulas by the easiest case: $q$ a root of unity in a characteristic $0$ field. In Table 9 we have recorded the dimensions of the simple modules for $E_q^n(2)$ when $\ell = 5$ for $n \leq 10$. Note that by Proposition \ref{ss} only the first 3 rows are identical to the corresponding rows in Table 8. The reader may check these numbers via (\refeq{simples via TL}) and Table 2 (but will need to compute the next $4$ rows in Table 2 in order to get the last two rows in Table 9).
 
\vfill \eject

Table 9. Dimensions of simple modules for $E_q^n(2)$ for $\ell = 5$ and $p = 0$.
\vskip .5 cm
\noindent
\begin{tabular}{ r| c  c c c c c c c c c c c c c c c c}
  
   & 0 & 2 & 4 & 6 & 8 & 10 & 12 & 14 & 16& 18 & 20  \\  \hline 
  
  0 & 1 &  \\ 
  1 &  &  1 \\ 
  2 & 1& 1 & 1   \\
  3 & 1 &2 & 2 & 1\\
4&2 &3 & 6& 3 & 1 & \\
5 &3 &5& 15&10 &3  &1  \\
6&5 &8 &40 &28&10&5&1 \\
7& 8&13&105 &78&28&21&6 & 1& \\
8&13&21&280&211&78&76&28&7&1\\
9 &21&34&750&569&211&257&103&36&8&1\\
10 &34&55&2025&1530&569&829&360&155&45&8&1\\

\end{tabular}
\end{examplecounter}

\vskip .5cm
We shall apply the above results to get information about the simple modules for the family of $BMW$-algebras. We start with defining the version of these algebras which we shall work with (there are several possible specializations of the $3$ parameter BMW-algebras, see e.g. \cite{TA}, \cite{LZ1}, \cite{RS1}). We start out working over the algebra of Laurent polynomials $A =\Z[v, v^{-1}]$ over $\Z$.

\begin{defn}
Let $r \in \Z_{\geq 0}$. Then $BMW_r(A)$ is the $A$-algebra with generators $\{g_i^{\pm 1}, e_i | i=1, 2, \cdots , r-1\}$ and relations
\begin{enumerate}
\item  $g_ig_j = g_jg_i$ if $|i-j| > 1$,
\item $g_ig_{i+1}g_i = g_{i+1}g_ig_{i+1}$ for $1 \leq i \leq r-1$,
\item $g_i - g_i^{-1} = (v^2 - v^{-2})(1-e_i)$  for $1 \leq i \leq r-1$,
\item $g_ie_i = e_ig_i =v^{-4}e_i$  for $1 \leq i \leq r-1$,
\item $e_ig_{i+1}^{\pm 1}e_i = v^{\pm 4}e_i$  for $1 \leq i \leq r-2$,
\item $e_ig_{i-1}^{\pm 1}e_i = v^{\pm 4}e_i$  for $2 \leq i \leq r-1$.
\end{enumerate}
\end{defn}

Let $k$ be a field of characteristic different from $2$ and suppose $q \in k$ is a root of unity with $\ell = \ord(q^2) > 2$ (i.e. the assumptions in the first paragraph of this section are satisfied). Then $k$ is an $A$-algebra via the map $v \mapsto q$ and we shall call $BMW_r(q) = BMW_r(A) \otimes_A k$ the BMW-algebra over $k$. (This is also a specialization of the Birman-Wenzl algebra studied by Rui and Si in \cite{RS1}, \cite{RS2}.)

Assuming that $\ell \neq 6$ it is proved (by a brute force argument) in \cite{TA} Section 4, that if $B_r$ is the braid group on $r-1$ strands then the  homomorphism from the group algebra $k[B_r]$ of $B_r$ into $E_q^r(2)$, which takes the $i$'th generator of $B_r$ into the $R$-matrix on $W^{\otimes r}$ operating on the $i$'th and $i+1$'st factor, is surjective. (If $k$ is replaced by the fraction field of $A$ then the corresponding statement was proved in \cite{LZ1} for all $m$.) This homomorphism factors through $BMW_r(q)$, i.e. we have
$$ E_q^r(2) \simeq BMW_r(q)/N,$$
where the kernel $N$ is known to be generated by a specified element  $\tilde \Phi_q \in BMW_4(q)$, see \cite{LZ2}, \cite{TA}. Hence the above results prove

\begin{cor} Assume $\ell  \neq 6$. Then  
the simple modules for $BMW_r(q)$ which are killed by $\tilde \Phi_q$ are $\{L_W^r(2m) | 0 \leq m \leq r \}$. Their dimensions are given by (\refeq{simples via TL}).
\end{cor}

\begin{examplecounter}
Consider the case where $\ell = 5$. When $p = 0$ Table 9 gives the dimensions of the simple modules $L_W^r(2m)$ for $m \leq r$ and $r = 0, 1, \cdots , 10$. 

If instead we take $p=3$ then Lemma \ref{small m mixed} shows that there are no changes as long as $r \leq 9$. However, as the reader can easily check from the results in Section 5, the numbers in the last row (the $r = 10$ case) change to $32, 55, 2025, 1530, 571, 828, 360, 155, 45, 8, 1$.  Note that only the simple $BMW_{10}(q)$-modules labeled by $0, 8$ and $10$ have changed dimension when we pass  from characteristic $0$ to characteristic $3$.

Finally, if we choose $p=2$ the dimensions of the simple $BMW_r(q)$ are the same as in characteristic $0$ for $r < 7$. In Table 10 we have recorded the dimensions when $ r = 7, 8, 9, 10.$
\vskip .5 cm
Table 10. Dimensions of simple modules for $E_q^n(2)$ for $\ell = 5$ and $p = 2$.
\vskip .5 cm
\noindent
\begin{tabular}{ r| c  c c c c c c c c c c c c c c c c}
  
   & 0 & 2 & 4 & 6 & 8 & 10 & 12 & 14 & 16& 18 & 20  \\  \hline

7& 8&13&104 &78&28&21&6 & 1& \\
8&13&21&273&210&78&76&27&7&1\\
9 &21&34&714&561&210&257&103&36&8&1\\
10 &34&55&1870&1485&561&829&360&155&45&8&1\\

\end{tabular}

\end{examplecounter}

\begin{remarkcounter}
To get all simple modules for $BMW_r(q)$ when $r>3$ one could explore the tilting theory for quantum groups for the symplectic and orthogonal Lie algebras, cf. Section 7 in \cite{AST2}.
\end{remarkcounter}

\vskip 2 cm
\end{document}